\documentclass[11pt,a4paper]{article}
\usepackage[utf8]{inputenc}

\usepackage{mathrsfs,amsthm,xcolor,verbatim,bbm,amsmath,amsfonts,amssymb,nicefrac,enumerate,hyperref,bm,mathtools,xparse,etoolbox}
\usepackage[margin=1in]{geometry}
\usepackage[capitalise]{cleveref} % clever reference

 \crefname{subsection}{Subsection}{Subsections}

\hypersetup{
    colorlinks,
    linkcolor={red!80!black},
    citecolor={green},
    urlcolor={blue!80!black}
}

\theoremstyle{plain}
\newtheorem{theorem}{Theorem}[section]
\newtheorem{lemma}[theorem]{Lemma}
\newtheorem{prop}[theorem]{Proposition}
\newtheorem{cor}[theorem]{Corollary}

\newtheorem{setting}[theorem]{Setting}

\theoremstyle{remark}

\theoremstyle{definition}

\DeclareMathAlphabet{\mathpzc}{OT1}{pzc}{m}{it}

\DeclareFontEncoding{LS1}{}{}
\DeclareFontSubstitution{LS1}{stix}{m}{n}
\DeclareMathAlphabet{\mathscr}{LS1}{stixscr}{m}{n}

\newcommand{\E}{\mathbb{E}}
\renewcommand{\P}{\mathbb{P}}

\newcommand{\R}{\mathbb{R}}
\newcommand{\N}{\mathbb{N}}

\newcommand{\w}[1]{\mathfrak{w}^{#1}}
\renewcommand{\b}[1]{\mathfrak{b}^{#1}}
\renewcommand{\v}[1]{\mathfrak{v}^{#1}}
\renewcommand{\c}[1]{\mathfrak{c}^{#1}}

\newcommand{\smallsum}{\textstyle\sum}

\newcommand{\cF}{\mathcal{F}}
\newcommand{\cG}{\mathcal{G}}

\newcommand{\cL}{\mathcal{L}}

\newcommand{\cV}{\mathcal{V}}

\newcommand{\scrN}{\mathscr{N}}

\newcommand{\fb}{\mathfrak{b}}
\newcommand{\fc}{\mathfrak{c}}

\newcommand{\fu}{\mathfrak{u}}
\newcommand{\fv}{\mathfrak{v}}
\newcommand{\fw}{\mathfrak{w}}

\renewcommand{\emptyset}{\varnothing}

\DeclarePairedDelimiter{\norm}{\lVert}{\rVert}
\DeclarePairedDelimiter{\abs}{\lvert}{\rvert}
\DeclarePairedDelimiter{\rbr}{(}{)}
\DeclarePairedDelimiter{\br}{[}{]}
\DeclarePairedDelimiter{\cu}{\{}{\}}

\renewcommand{\d}{\, \mathrm{d}}

\newcommand{\indicator}[1]{\mathbbm{1}_{\smash{#1}}}

\newcommand{\realization}[1] {\mathscr{N} ^{ #1  }}
\newcommand{\realapprox}[2]{\mathscr{N} ^{#1}_ {#2}}

\newcommand{\width}{H}

\ExplSyntaxOn

\bool_new:N \g_noteobserve

\NewDocumentCommand{\nobs}{}{
  \bool_if:nTF { \g_noteobserve } {
    \bool_gset_false:N \g_noteobserve 
    note~
  } {
    \bool_gset_true:N \g_noteobserve 
    observe~
  }
}

\NewDocumentCommand{\Nobs}{}{
  \bool_if:nTF { \g_noteobserve } {
    \bool_gset_false:N \g_noteobserve 
    Note~
  } {
    \bool_gset_true:N \g_noteobserve 
    Observe~
  }
}

\seq_new:N \g_cflist_loaded
\seq_new:N \g_cflist_pending

\NewDocumentCommand{\cfadd}{ m }
{
  \seq_if_in:NnF \g_cflist_loaded { #1 } {
    \seq_if_in:NnF \g_cflist_pending { #1 } {
      \seq_gput_right:Nn \g_cflist_pending { #1 }
    }
  }
}

\NewDocumentCommand{\cfconsiderloaded}{ m }{
  \seq_gput_right:Nn \g_cflist_loaded {#1}
}

\NewDocumentCommand{\cfremove}{ m }
{
  \seq_gremove_all:Nn \g_cflist_pending { #1 }
}

\NewDocumentCommand{\cfload}{ o }
{
  \seq_if_empty:NTF \g_cflist_pending {\unskip} {
    (cf.\ \cref{\seq_use:Nn \g_cflist_pending {,}})\IfValueTF{#1}{#1~}{\unskip}
    \seq_gconcat:NNN \g_cflist_loaded \g_cflist_loaded \g_cflist_pending
    \seq_gclear:N \g_cflist_pending
  }
}

\NewDocumentCommand{\cfclear} {} {
  \seq_gclear:N \g_cflist_loaded
  \seq_gclear:N \g_cflist_pending
}

\NewDocumentCommand{\cfout}{ o }
{
  \seq_if_empty:NTF \g_cflist_pending {\unskip} {
    (cf.\ \cref{\seq_use:Nn \g_cflist_pending {,}})\IfValueTF{#1}{#1~}{\unskip}
    \seq_gclear:N \g_cflist_pending
  }
}

\NewDocumentCommand{\ifnocf} { m } {
  \seq_if_empty:NT \g_cflist_pending { #1 }
}

\ExplSyntaxOff

\title{A proof of convergence for gradient descent in the training\\ of artificial neural networks for constant target functions}

\author{
Patrick Cheridito$^1$,
Arnulf Jentzen$^{2, 3}$,\\
Adrian Riekert$^4$,
and
Florian Rossmannek$^5$
\bigskip
\\
\small{$^1$ Department of Mathematics, ETH Zurich, }
\vspace{-0.1cm}\\
\small{Zurich, Switzerland, e-mail: \texttt{patrick.cheridito}\textcircled{\texttt{a}}\texttt{math.ethz.ch}}
\smallskip
\\
\small{$^2$ Faculty of Mathematics and Computer Science, University of M{\"u}nster,}
\vspace{-0.1cm}\\
\small{M{\"u}nster, Germany, e-mail: \texttt{ajentzen}\textcircled{\texttt{a}}\texttt{uni-muenster.de}}
\smallskip
\\
\small{$^3$ School of Data Science and Shenzhen Research Institute of Big Data,}
\vspace{-0.1cm}\\
\small{The Chinese University of Hong Kong, Shenzhen, China, e-mail: \texttt{ajentzen}\textcircled{\texttt{a}}\texttt{cuhk.edu.cn}}
\smallskip
\\
\small{$^4$ Faculty of Mathematics and Computer Science, University of M{\"u}nster,}
\vspace{-0.1cm}\\
\small{M{\"u}nster, Germany, e-mail: \texttt{ariekert}\textcircled{\texttt{a}}\texttt{uni-muenster.de}}
\smallskip
\\
\small{$^5$ Department of Mathematics, ETH Zurich, }
\vspace{-0.1cm}\\
\small{Zurich, Switzerland, e-mail: \texttt{florian.rossmannek}\textcircled{\texttt{a}}\texttt{math.ethz.ch}}
}

\date{\today}

\begin{document}

\maketitle

\begin{abstract}
   Gradient descent optimization algorithms are the standard ingredients that are used to train artificial neural networks (ANNs). Even though a huge number of numerical simulations indicate that gradient descent optimization methods do indeed convergence in the training of ANNs, until today there is no rigorous theoretical analysis which proves (or disproves) this conjecture. In particular, even in the case of the most basic variant of gradient descent optimization algorithms, the plain vanilla gradient descent method, it remains an open problem to prove or disprove the conjecture that gradient descent converges in the training of ANNs. In this article we solve this problem in the special situation where the target function under consideration is a constant function. More specifically, in the case of constant target functions we prove in the training of rectified fully-connected feedforward ANNs with one-hidden layer that the risk function of the gradient descent method does indeed converge to zero. Our mathematical analysis strongly exploits the property that the rectifier function is the activation function used in the considered ANNs. A key contribution of this work is to explicitly specify a Lyapunov function for the gradient flow system of the ANN parameters. This Lyapunov function is the central tool in our convergence proof of the gradient descent method.
\end{abstract}

\newpage

\tableofcontents

\newpage

\section{Introduction}

Gradient descent (GD) optimization schemes are the standard methods for the training of artificial neural networks (ANNs). Although a large number of numerical simulations hint that GD optimization methods do converge in the training of ANNs, in general there is no mathematical analysis in the scientific literature which proves (or disproves) the conjecture that GD optimization methods converge in the training of ANNs. 

Even though the convergence of GD optimization methods is still an open problem of research, there are several promising approaches in the scientific literature which attack this problem. In particular, we refer, e.g., to \cite{Bach2017, BachMoulines2013, BachMoulines2011} and the references mentioned therein for convergence results for GD optimization methods in the training of convex neural networks, we refer, e.g., to \cite{AllenzhuLiLiang2019, AllenzhuLiSong2019, ChizatBach2018, DuLeeLiWangZhai2019, DuZhaiPoczosSingh2018arXiv, EMaWu2020, JacotGabrielHongler2020, LiLiang2019, SankararamanDeXuHuangGoldstein2020, EChaoWu2018, ZouCaoZhouGu2019} and the references mentioned therein for convergence results for GD optimization methods for the training of ANNs in the so-called overparametrized regime, we refer, e.g., to \cite{AkyildizSabanis2021, FehrmanGessJentzen2020, LeiHuLiTang2020, LovasSabanis2020} and the references mentioned therein for abstract convergence results for GD optimization methods which do not assume convexity of the considered objective functions, we refer, e.g., to \cite{Hanin2018, HaninRolnick2018, LuShinSuKarniadakis2020, ShinKarniadakis2020} and the references mentioned therein for results on the effect of initialization in the training of ANNs, and we refer, e.g., to \cite{CheriditoJentzenRossmannek2020, JentzenvonWurstemberger2020, LuShinSuKarniadakis2020} and the references mentioned therein for lower bounds and divergence results for GD optimization methods. For more detailed overviews and further references on GD optimization schemes we also refer, e.g., to \cite{Ruder2017overview}, \cite[Section 1]{JentzenKuckuckNeufeldVonWurstemberger2021}, and \cite[Section 1.1]{FehrmanGessJentzen2020}.

A key idea of this work is to attack this challenging open problem of convergence of GD optimization methods in the training of ANNs in the situation of very special target functions: Our program is to first establish convergence in the case of constant target functions, thereafter, to prove convergence in the case of affine linear target functions, thereafter, to consider suitable continuous piecewise affine linear target functions, and, finally, to pass to the limit of general continuous target functions. In particular, the central contribution of this work is to solve this problem in the case of constant target functions. More formally, the main result of this article (see \cref{theo:gd:loss} in \cref{subsection:theorem:gd} below) proves that the risk function of the standard GD process converges to zero in the training of fully-connected rectified feedforward ANNs with one input, one output, and one hidden layer in the special situation where the target function under consideration is a constant function and where the input data is continuous uniformly distributed. In the next result, \cref{theo:intro}, we illustrate the findings of this work in more detail within this introductory section. Below \cref{theo:intro} we add several 
explanatory comments regarding the statement of and the mathematical objects in 
\cref{theo:intro} and we also highlight the key ideas of the proof of \cref{theo:intro}.

\begin{theorem} \label{theo:intro}
Let $\width \in \N$, $\alpha \in \R$,
$\gamma \in (0 , \infty ) $,
let $\norm{\cdot} \colon \R^{3 \width + 1 } \to [0, \infty)$ satisfy for all $\phi = ( \phi_1 , \ldots, \phi_{ 3 \width + 1 } ) \in \R^{3 \width + 1}$ that $\norm{ \phi } = \br{ \sum_{i=1}^{3 \width + 1 } \abs*{ \phi_i } ^2 } ^{ 1 / 2 }$,
let $\sigma_r \colon \R \to \R$, $r \in[1 , \infty]$, satisfy for all $r \in [1 , \infty)$, $x \in \R$ that $\sigma_r ( x ) = r^{-1} \ln \rbr{  1 + r^{-1} e^{r x }  }$ and $\sigma_\infty ( x ) = \max \{ x , 0 \}$,
let $\scrN_ r = (\realapprox{\phi}{r})_{\phi \in \R^{3 \width + 1 } } \colon \R^{3 \width + 1} \to C(\R , \R)$, $r \in [1 , \infty]$, and $\cL _ r \colon \R^{3 \width + 1 } \to \R$, $r \in [1 , \infty]$,
satisfy for all $r \in [1 , \infty]$, $\phi = ( \phi_1 , \ldots, \phi_{ 3 \width + 1 } ) \in \R^{3 \width + 1}$, $x \in \R$ that $\realapprox{\phi}{r} (x) = \phi_{3 \width + 1 } + \sum_{j=1}^\width \phi_{2 \width + j} \sigma_r (\phi_j x + \phi_{\width + j} )$
and $\cL_r(\phi) = \int_0^1 (\realapprox{\phi}{r} (y) - \alpha )^2 \d y$,
let $\cG = ( \cG_1,  \ldots, \cG_{3 \width + 1} ) \colon \R^{3 \width + 1} \to \R^{3 \width + 1}$ satisfy for all
$\phi \in  \{ \varphi \in \R^{3 \width + 1} \colon   ((\nabla \cL_r ) ( \varphi ) )_{r \in \N } \text{ is convergent}  \}$ that $\cG ( \phi ) = \lim_{r \to \infty} (\nabla \cL_r ) ( \phi )$,
and let $\Theta = (\Theta_n)_{n \in \N_0}  \colon \N_0 \to \R^{3 \width + 1}$ satisfy for all $n \in \N_0$ that $\Theta_{n+1} = \Theta_n - \gamma \cG ( \Theta_n)$ and  $\gamma \leq (4 \norm{\Theta_0}  + 6 \abs*{ \alpha } + 2 )^{-2}$. Then 
\begin{enumerate} [(i)]
    \item \label{theo:intro:item1} it holds for all $\phi \in \{ \varphi \in \R^{3 \width + 1 } \colon \cL_\infty \text{ is differentiable at } \varphi \}$ that $(\nabla \cL_\infty) ( \phi ) = \cG ( \phi)$,
    \item \label{theo:intro:item2} it holds that $\sup_{n \in \N_0} \norm{\Theta_n} < \infty$, and
    \item \label{theo:intro:item3} it holds that $\limsup_{n \to \infty} \cL_\infty (\Theta_n) = 0$.
\end{enumerate}
\end{theorem}
Item \eqref{theo:intro:item1} in \cref{theo:intro} is a direct consequence of \cref{cor:loss:differentiable} below and items
\eqref{theo:intro:item2} and \eqref{theo:intro:item3} in \cref{theo:intro} are direct consequences of \cref{cor:gd:main} below. \cref{cor:gd:main}, in turn, follows from \cref{theo:gd:loss}, which is the main result of this article.

Let us next add a few comments regarding the mathematical objects appearing in \cref{theo:intro}. 
In \cref{theo:intro} we study the training of ANNs with one input, one output, and one hidden layer. The natural number $\width \in \N$ in \cref{theo:intro} specifies the number of neurons on the hidden layer (the dimension of the hidden layer) in the ANN. \cref{theo:intro} proves that the risk function of GD converges to zero in the special situation where the input data is continuous uniformly distributed and where the target function under consideration is a constant function. The real number $\alpha \in \R$ is precisely this constant with which the target function is assumed to coincide. The real number $\gamma \in (0,\infty)$ in \cref{theo:intro} specifies the learning rate of the GD method. 

In \cref{theo:intro} we consider fully-connected feedforward ANNs with $1$ neuron on the input layer, $\width $ neurons on the hidden layer, and $1$ neuron on the output layer. Therefore, the considered ANNs have precisely $2 \width$ weights, $\width + 1$ biases, and
$2 \width + \width + 1 = 3 \width + 1$ ANN parameters overall. The function $\norm{\cdot} \colon \R^{3 \width + 1 } \to \R$ in \cref{theo:intro} is nothing else but the standard norm on the space $\R^{ 3 \width + 1 }$ of ANN parameters. 

In \cref{theo:intro} we study the training of ANNs with the rectifier function
$\R \ni x \mapsto \sigma_{ \infty }( x ) = \max\{ x, 0 \} \in \R$ as the activation function. Since the rectifier function $\sigma_{ \infty } \colon \R \to \R$ 
in \cref{theo:intro} is not differentiable at 0, we have that the associated risk function also fails to be differentiable at some points in the ANN parameter space $\R^{ 3 \width + 1 }$. In view of this, one needs to carefully choose the values for the driving gradient field in the GD optimization method at the points in the ANN parameter space $\R^{ 3 \width + 1 }$ where the risk function is not differentiable. We accomplish this by approximating the rectifier function and the corresponding risk function through regularized versions of these functions. More formally, in \cref{prop:relu:approximation} in \cref{subsection:relu:approx} below we show that the functions $\sigma_r \colon \R \to \R$, $r \in [1,\infty]$, in \cref{theo:intro} satisfy that for all $x \in \R$, $y \in \R \backslash \{ 0 \}$ it holds that $\limsup_{ r \to \infty } | \sigma_r(x) - \sigma_{ \infty }(x) |$ = 0 
and $\limsup_{ r \to \infty } | ( \sigma_r )'(y) - ( \sigma_{ \infty } )'(y) |$ = 0. \Nobs that for all $r \in [1,\infty)$ it holds that $\sigma_r \in C^{ \infty } ( \R , \R)$. 

The functions $\scrN_r \colon \R^{3 \width + 1 }  \to C ( \R , \R)$, $r \in [1,\infty]$, in \cref{theo:intro} describe the realization functions of the considered ANNs. More formally, \nobs that for every $r \in [1,\infty]$ and every $\phi = ( \phi_1, \ldots , \phi_{ 3 \width + 1 } ) \in \R^{ 3 \width + 1 }$ we have that the 
function $\R \ni x \mapsto \realapprox{ \phi }{r}( x) \in \R$ is the realization function 
associated to the ANN with the activation function $\sigma_r \colon \R \to \R$ and the parameter vector $\phi = ( \phi_1, \ldots , \phi_{ 3 \width + 1 } )$. In particular, \nobs that for every 
ANN parameter vector $\phi \in \R^{ 3 \width + 1 }$ we have that 
$\R \ni x \mapsto \realapprox{ \phi }{\infty} ( x) \in \R$ is the realization function associated to the rectified ANN with the parameter vector $\phi$.

The process $\Theta = ( \Theta_n )_{ n \in \N_0 } \colon \N_0 \to \R^{3 \width + 1 }$ in \cref{theo:intro} is the GD process with constant learning rate $\gamma$. \Nobs that the learning rate $\gamma$ in \cref{theo:intro} is assumed to be sufficiently small in the sense that $\gamma \leq (4 \norm{\Theta_0}  + 6 \abs*{ \alpha } + 2 )^{-2}$. Under this assumption, \cref{theo:intro} reveals that the risk of the GD process $\cL_{ \infty }( \Theta_n )$, $n \in \N_0$, does indeed converge to zero as the number of GD steps $n$ increases to infinity.

Let us also add a few comments on the proof of \cref{theo:intro}. A key new observation of this article is the fact that in the situation of \cref{theo:intro} we have that the function
\begin{equation} \label{eq:intro:lyapunov}
 \R^{ 3 \width + 1 } \ni ( \phi_1, \ldots, \phi_{ 3 \width + 1 } ) \mapsto  \rbr[\big]{\smallsum_{i=1}^{3 \width + 1 } \abs{\phi_i}^2 } + ( \phi_{3 \width + 1 } - 2 \alpha ) ^2 \in \R
\end{equation}
is a Lyapunov function for the gradient flow system of the ANN parameters. 
We refer to item \eqref{prop:lyapunov:gradient:item3} in \cref{prop:lyapunov:gradient} in \cref{subsection:lyapunov} and 
\cref{lem:flow:lyapunov} in \cref{subsection:ito:lyapunov} for the proof of this statement. 
In addition, in \cref{lem:vthetan:decreasing} in \cref{subsection:gd:lyapunov} we
show that the function in \eqref{eq:intro:lyapunov} is also a Lyapunov function for the time-discrete GD processes if the learning rate is sufficiently small. We also would like to emphasize that the term $( \phi_{3 \width + 1 } - 2 \alpha ) ^2$ in \eqref{eq:intro:lyapunov} is essential for the function in \eqref{eq:intro:lyapunov} to serve as a Lyapunov function. In particular, we would like to point out that the function $\R^{3 \width + 1 } \ni \phi \mapsto \norm{\phi} ^2 \in \R$ fails to be a Lyapunov function for the gradient flow system of the ANN parameters.

The remainder of this article is structured as follows. In \cref{section:risk:regularity} we 
present the mathematical framework which we use to study the considered GD processes 
and we also establish several regularity properties for the considered risk functions 
and their gradients.
In \cref{section:gradientflow} we use the findings from \cref{section:risk:regularity} to establish that the risks of the considered time-continuous gradient flow processes converge to zero.
In \cref{section:gradientdescent} we prove that the risks of the considered time-discrete GD processes converge to zero.
The key ingredient in our convergence proofs for gradient flow and GD processes in Sections \ref{section:gradientflow} and \ref{section:gradientdescent} are suitable a priori estimates (which we achieve by means of the Lyapunov function in \eqref{eq:intro:lyapunov} above) for the gradient flow processes (see \cref{lem:flow:lyapunov} in \cref{subsection:ito:lyapunov}) and the GD processes (see \cref{lem:vthetan:decreasing} in \cref{subsection:gd:lyapunov}). In \cref{section:apriori:gen}
we derive -- to stimulate further research activities -- related a priori bounds in the case of general target functions.

\section{Regularity properties of the risk functions and their gradients}
\label{section:risk:regularity}

In \cref{section:risk:regularity} we present in \cref{setting:const} the mathematical framework which we use to study the considered GD processes and we also establish several regularity results for the considered risk functions and their gradients. Most notably, we establish in Propositions \ref{prop:lyapunov:norm} and \ref{prop:lyapunov:gradient} in \cref{subsection:lyapunov} below that the gradient flow system for the ANN parameters in \cref{setting:const} admits an appropriate Lyapunov function. In particular, in item \eqref{prop:lyapunov:gradient:item3} in \cref{prop:lyapunov:gradient} we prove that the function $V \colon \R^{3 \width + 1 } \to  \R$ in \cref{setting:const} serves as a Lyapunov function.

We also note that the results in \cref{prop:relu:approximation} in \cref{subsection:relu:approx}, in \cref{lem:interchange} in \cref{subsection:risk:differentiable}, and in \cref{cor:interchange} in \cref{subsection:risk:differentiable} are all well-known in the literature and we include in this section detailed proofs for \cref{prop:relu:approximation}, \cref{lem:interchange}, and \cref{cor:interchange} only for completeness.

\subsection{Mathematical description of rectified artificial neural networks}

\begin{setting} \label{setting:const} 
Let $\width \in \N$, $\alpha \in \R$,
let $\fw  = (( \w{\phi} _ 1 , \ldots, \w{\phi} _ \width ))_{ \phi \in \R^{3 \width + 1}} \colon \R^{3 \width + 1} \to \R^{\width}$,
$\fb =  (( \b{\phi} _ 1 , \ldots, \b{\phi} _ \width ))_{ \phi \in \R^{3 \width + 1}} \colon \R^{3 \width + 1} \to \R^{\width}$,
$\fv = (( \v{\phi} _ 1 , \ldots, \v{\phi} _ \width ))_{ \phi \in \R^{3 \width + 1}} \colon \R^{3 \width + 1} \to \R^{\width}$, and
$\fc = (\c{\phi})_{\phi \in \R^{3 \width + 1 }} \colon \R^{3 \width + 1} \to \R$
 satisfy for all $\phi  = ( \phi_1 ,  \ldots, \phi_{3 \width + 1}) \in \R^{3 \width + 1}$, $j \in \{1, 2, \ldots, \width \}$ that $\w{\phi}_j = \phi_j$, $\b{\phi}_j = \phi_{\width + j}$, 
$\v{\phi}_j = \phi_{2\width + j}$, and $\c{\phi} = \phi_{3 \width + 1}$,
let $\sigma_r \colon \R \to \R$, $r \in[1 , \infty]$, satisfy for all $r \in [1 , \infty)$, $x \in \R$ that $\sigma_r ( x ) = r^{-1} \ln \rbr{  1 + r^{-1} e^{r x }  }$ and $\sigma_\infty ( x ) = \max \{ x , 0 \}$,
let $\scrN_ r = (\realapprox{\phi}{r})_{\phi \in \R^{3 \width + 1 } } \colon \R^{3 \width + 1} \to C(\R , \R)$, $r \in [1 , \infty]$, and $\cL _ r \colon \R^{3 \width + 1 } \to \R$, $r \in [1 , \infty]$,
satisfy for all $r \in [1 , \infty]$, $\phi \in \R^{3 \width + 1}$, $x \in \R$ that $\realapprox{\phi}{r} (x) = \c{\phi} + \sum_{j=1}^\width \v{\phi}_j \sigma_r (\w{\phi}_j x + \b{\phi}_j )$
and $\cL_r(\phi) = \int_0^1 (\realapprox{\phi}{r} (y) - \alpha )^2 \d y$,
let $\cG = ( \cG_1,  \ldots, \cG_{3 \width + 1} ) \colon \R^{3 \width + 1} \to \R^{3 \width + 1}$ satisfy for all
$\phi \in  \{ \varphi \in \R^{3 \width + 1} \colon   ((\nabla \cL_r ) ( \varphi ) )_{r \in \N } \text{ is convergent} \}$ that $\cG ( \phi ) = \lim_{r \to \infty} (\nabla \cL_r ) ( \phi )$,
let $\norm{ \cdot } \colon \rbr*{  \bigcup_{n \in \N} \R^n  } \to [0, \infty)$ and $\langle \cdot , \cdot \rangle \colon \rbr*{  \bigcup_{n \in \N} (\R^n \times \R^n )  } \to \R$ satisfy for all $n \in \N$, $x=(x_1, \ldots, x_n), y=(y_1, \ldots, y_n ) \in \R^n $ that $\norm{ x } = [ \sum_{i=1}^n \abs*{ x_i } ^2 ] ^{1/2}$ and $\langle x , y \rangle = \sum_{i=1}^n x_i y_i$,
and let $I_j^\phi \subseteq \R$, $\phi \in \R^{3 \width + 1 }$, $j \in \{1, 2, \ldots, \width \}$, and $V \colon \R^{3 \width + 1} \to \R$ satisfy for all 
$\phi \in \R^{3 \width +1}$, $j \in \{1, 2, \ldots, \width \}$ that $I_j^\phi = \{ x \in [0,1] \colon \w{\phi}_j x + \b{\phi}_j > 0 \}$ and
 $V(\phi) = \norm{ \phi } ^2 + ( \c{\phi} -  2 \alpha) ^2$ .
\end{setting}

\subsection{Smooth approximations of the rectifier function}
\label{subsection:relu:approx}

\begin{prop} \label{prop:relu:approximation} 
Let $\sigma_r \colon \R \to \R$, $r \in [1, \infty]$, satisfy for all $r \in [1 , \infty)$, $x \in \R$ that $\sigma_r ( x ) = r^{-1} \ln \rbr{  1 + r^{-1} e^{r x }  }$ and $\sigma_\infty ( x ) = \max \{ x , 0 \}$.
Then
\begin{enumerate} [(i)] 
   \item \label{prop:relu:approx:item1} it holds for all $r \in [1 , \infty)$ that $\sigma_r \in C^\infty ( \R , \R )$,
    \item \label{prop:relu:approx:item2} it holds for all $r \in [1 , \infty)$, $x \in \R$ that $0  < \sigma_r ( x ) < \sigma_\infty (x) + 1$,
    \item \label{prop:relu:approx:item3} it holds for all $x \in \R$ that $\limsup_{r \to \infty} \abs*{ \sigma_r(x) - \sigma_\infty (x) } = 0$,
    \item \label{prop:relu:approx:item4} it holds for all $r \in [1 , \infty)$, $x \in \R$ that $0 < (\sigma_r)'(x) < 1$, and
    \item \label{prop:relu:approx:item5} it holds for all $x \in \R$ that $\limsup_{r \to \infty} \abs*{ (\sigma_r)'(x) - \indicator{(0, \infty)} ( x ) } = 0$.
\end{enumerate}
\end{prop}

\begin{proof} [Proof of \cref{prop:relu:approximation}]
\Nobs that the fact that $(\R \ni x \mapsto e^x \in \R) \in C^\infty ( \R , \R)$, the fact that $((0, \infty) \ni x \mapsto \ln (x) \in \R) \in C^\infty ( (0, \infty) , \R)$, and the chain rule prove item \eqref{prop:relu:approx:item1}. 
Next \nobs that for all $r \in [1 , \infty)$, $x \in (- \infty , 0 ]$ it holds that $1 < 1 + r^{-1} e^{rx} \leq 2$ and therefore 
\begin{equation} \label{eq:relu:approximation:1}
    0 < \sigma_r( x ) \leq r^{-1} \ln ( 2) < r^{-1} \leq 1 = \sigma_\infty (x) + 1.
\end{equation}
This establishes for all $x \in (- \infty , 0 ]$ that $\limsup_{r \to \infty} \abs*{ \sigma_r(x) - \sigma_\infty (x) } \leq \limsup_{r \to \infty} (r^{-1}) \allowbreak = 0 $.
Moreover, \nobs that for all $r \in [1 , \infty)$, $x \in (0, \infty)$ it holds that 
\begin{equation}
\label{eq:reluapprox:2}
0 = r^{-1} \ln ( 1 ) < \sigma _ r ( x ) \leq r^{-1} \ln ( 2 e ^{ r x } ) = x + r ^{-1} \ln ( 2 ) < x + 1 = \sigma_\infty (x) + 1.
\end{equation}
This and \eqref{eq:relu:approximation:1} prove item \eqref{prop:relu:approx:item2}.
In addition, \nobs that for all $r \in [1 , \infty)$, $x \in (0, \infty)$ it holds that $\sigma _ r ( x ) \geq r^{-1} \ln ( r^{-1} e^{r x } ) = x - r^{-1} \ln (r)$.
Combining this with \eqref{eq:reluapprox:2} demonstrates for all $x \in (0, \infty)$ that 
\begin{equation}
\begin{split} 
\limsup_{r \to \infty} \abs*{ \sigma _ r ( x ) - \sigma _ \infty ( x ) } 
&=  \limsup_{r \to \infty}  \abs*{ \sigma _ r ( x ) -  x } \\
&\leq \limsup_{r \to \infty} \br*{ \max \cu*{ r^{-1} \ln(2) , r^{-1} \ln (r) } } = 0,
\end{split}
\end{equation}
which completes the proof of item \eqref{prop:relu:approx:item3}. 
To prove item \eqref{prop:relu:approx:item4}, \nobs that the chain rule implies for all $r \in [1 , \infty)$, $x \in \R$ that
\begin{equation} \label{eq:reluapprox:3}
    (\sigma_r ) ' ( x ) = \frac{1}{r} \br*{ \frac{e^{ r x }}{1 + r^{-1} e^{r x}} } = \frac{1}{1 + r e^{-r x }}.
\end{equation}
This demonstrates for all $r \in [1 , \infty)$, $x \in \R$ that $0 < (\sigma_r) ' ( x ) < 1$, which establishes item \eqref{prop:relu:approx:item4}. Next \nobs that \eqref{eq:reluapprox:3} and the fact that for all $r \in [1 , \infty)$, $x \in (- \infty , 0]$ it holds that $e^{-r x } \geq 1$ show that for all $r \in [1 , \infty)$, $x \in (- \infty , 0]$ it holds that $(\sigma_r) ' ( x ) \leq \frac{1}{1+r}$. On the other hand, \nobs that for all $x \in (0, \infty)$ we have that $\lim_{r \to \infty} (r e^{-r x }) = 0$ and thus $\lim_{r \to \infty} (\sigma_r) ' ( x ) = 1$. This establishes item \eqref{prop:relu:approx:item5}. The proof of \cref{prop:relu:approximation} is thus complete.
\end{proof}

\subsection{Differentiability properties of the risk functions}
\label{subsection:risk:differentiable}

\begin{prop} \label{prop:limit:lr}
Assume \cref{setting:const} and let $\phi = (w_1, \ldots, w_{\width}, b_1, \ldots, b_{\width}, v_1, \ldots, \allowbreak v_{\width}, c) \in \R^{3 \width + 1}$. Then 
\begin{enumerate} [(i)]
\item \label{prop:limit:lr:1} it holds for all $r \in [1 , \infty)$ that $\cL _ r \in C^1 ( \R^{3 \width + 1}, \R)$,
    \item \label{prop:limit:lr:2} it holds for all $r \in [1 , \infty)$, $j \in \{1, 2, \ldots, \width \}$ that
\begin{equation} \label{eq:approx:loss:gradient}
    \begin{split}
        \rbr[\big]{  \tfrac{\partial  }{ \partial w_j} \cL_r } ( \phi ) &= 2 v_j \int_0^1 x \br*{  (\sigma_r )' ( w_j x + b_j) } ( \realapprox{\phi}{r}(x) - \alpha) \d x, \\
         \rbr[\big]{  \tfrac{\partial }{ \partial b_j}  \cL_r } ( \phi ) &= 2 v_j \int_0^1  \br*{  (\sigma_r) ' ( w_j x + b_j ) } ( \realapprox{\phi}{r}(x) - \alpha) \d x, \\
          \rbr[\big]{  \tfrac{\partial }{ \partial v_j}  \cL_r } ( \phi ) &= 2  \int_0^1 \br*{  \sigma_r ( w_j x + b_j ) } ( \realapprox{\phi}{r}(x) - \alpha) \d x,  \\
           \rbr[\big]{  \tfrac{\partial  }{ \partial c} \cL_r } ( \phi ) &= 2  \int_0^1 ( \realapprox{\phi}{r}(x) - \alpha) \d x,
    \end{split}
\end{equation}
\item \label{prop:limit:lr:3} it holds that $\limsup_{r \to \infty} \abs{ \cL_r ( \phi ) - \cL_\infty ( \phi) } = 0$,
\item \label{prop:limit:lr:4} it holds that $\limsup_{r \to \infty } \norm{ ( \nabla \cL _ r ) ( \phi ) - \cG ( \phi ) }  = 0$, and
\item \label{prop:limit:lr:5} it holds for all $j \in \{1, 2, \ldots, \width \}$ that
\begin{equation} \label{eq:loss:gradient}
\begin{split}
        \cG_j ( \phi) &= 2v_j \int_{I_j^\phi} x ( \realapprox{\phi}{\infty} (x) - \alpha ) \d x, \\
        \cG_{\width + j} ( \phi) &= 2 v_j \int_{I_j^\phi} (\realapprox{\phi}{\infty} (x) - \alpha ) \d x, \\
        \cG_{2 \width + j} ( \phi) &= 2 \int_0^1 [\sigma_\infty (w_j x + b_j) ] ( \realapprox{\phi}{\infty}(x) - \alpha ) \d x, \\
        \cG_{3 \width + 1} ( \phi) &= 2 \int_0^1 (\realapprox{\phi}{\infty} (x) - \alpha ) \d x.
        \end{split}
\end{equation}
\end{enumerate}
\end{prop}
\begin{proof} [Proof of \cref{prop:limit:lr}]
\Nobs that \cref{prop:relu:approximation}, the chain rule, and the dominated convergence theorem establish items \eqref{prop:limit:lr:1} and \eqref{prop:limit:lr:2}.
Next \nobs that \cref{prop:relu:approximation} demonstrates for all $x \in [0, 1]$ that $\lim_{r \to \infty} ( \realapprox{\phi}{r} ( x ) - \alpha ) = \realapprox{\phi}{\infty} ( x ) - \alpha$. 
Furthermore, \nobs that \cref{prop:relu:approximation} shows that for all $x \in [0,1]$, $r \in [1 , \infty)$ it holds that
\begin{equation} \label{proof:limit:lr:eq1}
\begin{split}
     \abs{ \realapprox{\phi}{r} ( x ) - \alpha }
   &\leq \abs{ \alpha } + \abs{ c } + \smallsum_{j=1}^\width | v_j | ( \sigma_\infty ( w_j x + b_j ) + 1 ) \\
   &\leq | \alpha | + | c | + \smallsum_{j=1}^\width | v_j | ( | w_j | + | b_j | + 1 ).
   \end{split}
\end{equation}
The dominated convergence theorem hence proves that $\lim_{r \to \infty} \cL_r ( \phi ) = \cL_\infty ( \phi)$, which establishes item \eqref{prop:limit:lr:3}. Moreover, \nobs that the fact that $\forall \, x \in [0,1] \colon \lim_{r \to \infty} ( \realapprox{\phi}{r} ( x ) - \alpha ) = \realapprox{\phi}{\infty} ( x ) - \alpha$, \eqref{proof:limit:lr:eq1}, and the dominated convergence theorem prove that
\begin{equation} \label{limit:lr:eq2}
    \lim_{r \to \infty}  \br*{  \rbr[\big]{ \tfrac{\partial  }{ \partial c} \cL_r } ( \phi ) }
    = 2 \int_0^1 (\realapprox{\phi}{\infty} (x) - \alpha ) \d x.
\end{equation}
Next \nobs that \cref{prop:relu:approximation} shows for all $x \in [0,1]$, $j \in \{1, 2, \dots, \width \}$ that
\begin{equation}
    \begin{split}
        \lim_{r \to \infty}  \br*{ x  \br*{  (\sigma_r) ' ( w_j x + b_j ) } ( \realapprox{\phi}{r}(x) - \alpha) }
        &= x ( \realapprox{\phi}{\infty} ( x ) - \alpha ) \indicator{(0 , \infty ) }  ( w_j x + b_j) \\
        &= x ( \realapprox{\phi}{\infty} ( x ) - \alpha )\indicator{I_j^\phi} ( x )
    \end{split}
\end{equation}
and
\begin{equation}
\begin{split}
     \lim_{r \to \infty} \br*{   [(\sigma_r) ' ( w_j x + b_j )] ( \realapprox{\phi}{r}(x) - \alpha) } 
     &=  ( \realapprox{\phi}{\infty} ( x ) - \alpha ) \indicator{(0 , \infty ) } ( w_j x + b_j ) \\
     &=  ( \realapprox{\phi}{\infty} ( x ) - \alpha )  \indicator{I_j^\phi} ( x ).
     \end{split}
\end{equation}
Furthermore, \nobs that \cref{prop:relu:approximation} and \eqref{proof:limit:lr:eq1} prove that for all $r \in [1 , \infty)$, $x \in [0,1]$, $j \in \{1, 2, \ldots, \width \}$ it holds that
\begin{equation}
\begin{split}
    &\abs[\big]{  x [(\sigma_r) ' ( w_j x + b_j )] ( \realapprox{\phi}{r}(x) - \alpha) } \\
    &\leq   \abs[\big]{  [(\sigma_r ) ' ( w_j x + b_j )] ( \realapprox{\phi}{r}(x) - \alpha) }\\
    &\leq| \realapprox{\phi}{r} ( x ) - \alpha |
   \leq | \alpha | + | c | + \smallsum_{j=1}^\width | v_j | ( | w_j | + | b_j | + 1 ).
   \end{split}
\end{equation}
The dominated convergence theorem hence proves for all $j \in \{1, 2, \ldots, \width \}$ that
\begin{equation} \label{limit:lr:eq3}
    \lim_{r \to \infty} \br[\big]{  \rbr[\big]{  \tfrac{\partial  }{ \partial w_j} \cL_r } ( \phi ) }
    =  2 v_j \int_0^1 x ( \realapprox{\phi}{\infty} ( x ) - \alpha )  \indicator{I_j^\phi} ( x ) \d x 
    = 2v_j \int_{I_j^\phi} x ( \realapprox{\phi}{\infty} (x) - \alpha ) \d x
\end{equation}
and
\begin{equation} \label{limit:lr:eq4}
    \lim_{r \to \infty} \br[\big]{  \rbr[\big]{  \tfrac{\partial  }{ \partial b_j} \cL_r } ( \phi ) }
    =2 v_j   \int_0^1   ( \realapprox{\phi}{\infty} ( x ) - \alpha ) \indicator{I_j^\phi} ( x ) \d x 
    =  2v_j \int_{I_j^\phi } ( \realapprox{\phi}{\infty} (x) - \alpha ) \d x .
\end{equation}
Moreover, \nobs that \cref{prop:relu:approximation} and \eqref{proof:limit:lr:eq1} show that for all $r \in [1 , \infty)$, $x \in [0,1]$, $j \in \{1, 2, \ldots, \width \}$ it holds that
\begin{equation}
    \lim_{r \to \infty} \br*{  [\sigma_r ( w_j x + b_j )] ( \realapprox{\phi}{r}(x) - \alpha) } = [\sigma _\infty ( w_j x + b_j )] ( \realapprox{\phi}{\infty}(x) - \alpha)
\end{equation}
and
\begin{equation}
\begin{split}
    &\abs[\big]{  [\sigma_r ( w_j x + b_j )] ( \realapprox{\phi}{r}(x) - \alpha) } \\
    & \leq (\sigma_\infty ( w_j x + b_j ) + 1 ) | \realapprox{\phi}{r}(x) - \alpha | 
   \\
   &\leq ( 1 + | w_j  | + | b_j | ) | \realapprox{\phi}{r}(x) - \alpha | \\ &
    \leq ( 1 + | w_j | + | b_j | ) \rbr*{  | \alpha | + | c | + \smallsum_{j=1}^\width | v_j | ( | w_j | + | b_j | + 1 )  }.
\end{split}
\end{equation}
This and the dominated convergence theorem demonstrate for all $j \in \{1, 2, \ldots, \width \}$ that 
\begin{equation} 
\lim_{r \to \infty} \br[\big]{  \rbr[\big]{  \tfrac{\partial  }{ \partial v_j} \cL_r } ( \phi ) } = 2 \int_0^1 [\sigma_\infty (w_j x + b_j)] ( \realapprox{\phi}{\infty}(x) - \alpha ) \d x.
\end{equation}
Combining this, \eqref{limit:lr:eq2}, \eqref{limit:lr:eq3}, and \eqref{limit:lr:eq4} establishes items \eqref{prop:limit:lr:4} and \eqref{prop:limit:lr:5}. The proof of \cref{prop:limit:lr} is thus complete.
\end{proof}

\begin{lemma} \label{lem:interchange}
Let $\fu \in \R$, $\fv \in (\fu , \infty)$, let $f \colon \R \times [\fu , \fv] \to \R$ be locally Lipschitz continuous, let $F \colon \R \to \R$ satisfy for all $x \in \R$ that
\begin{equation}
    F(x) = \int_\fu ^\fv f(x,s) \d s,
\end{equation}
let $x \in \R$, let $E \subseteq [\fu , \fv]$ be measurable, assume $\int_{[\fu , \fv] \backslash E } 1 \d s = 0 $, and assume for all $s \in E$ that $\R \ni v \mapsto f ( v , s ) \in \R$ is differentiable at $x$. Then
\begin{enumerate} [(i)]
    \item it holds that $F$ is differentiable at $x$ and
    \item it holds that
    \begin{equation}
        F'(x) = \int_E \rbr[\big]{\tfrac{\partial}{\partial x} f }  ( x , s ) \d s.
    \end{equation}
\end{enumerate}
\end{lemma}
\begin{proof} [Proof of \cref{lem:interchange}]
 \Nobs that the assumption that $\int_{[\fu , \fv] \backslash E } 1 \d s = 0 $ ensures that for all $h \in \R \backslash \{ 0 \}$ we have that
\begin{equation} \label{lem:interchange:eq1}
 h^{-1} [ F(x+h) - F(x)] = \int_\fu^\fv h^{-1} [ f(x+h , s ) - f(x,s)] \d s = \int_E h^{-1} [ f(x+h , s ) - f(x,s)] \d s.
\end{equation}
Next \nobs that the assumption that for all $s \in E$ it holds that $\R \ni v \mapsto f ( v , s) \in \R$ is differentiable at $x$ implies that for all $s \in E$ it holds that
\begin{equation} \label{lem:interchange:eq2}
    \lim\nolimits_{\abs{h} \searrow  0} \rbr*{ h^{-1} \br{ f( x+h, s) - f(x, s)}} = \rbr[\big]{\tfrac{\partial}{\partial x} f }(x, s ).
\end{equation}
Furthermore, \nobs that the assumption that $f$ is locally Lipschitz continuous ensures that for all $\delta \in (0, \infty)$ there exists $C \in (0, \infty)$ such that for all $h \in [-\delta, \delta] \backslash \{0 \}$, $s \in [ \fu , \fv]$ we have that $| h^{-1} \br{ f( x+h, s) - f(x, s)}| \leq C$. Combining this, \eqref{lem:interchange:eq1}, \eqref{lem:interchange:eq2}, and the dominated convergence theorem establishes that
\begin{equation}
    \begin{split}
        \lim\nolimits_{\abs{h} \searrow 0} \rbr*{ h^{-1}  \br{F(x+h) - F(x)}} 
        &= \int_E \br*{ \lim\nolimits_{\abs{h} \searrow 0}  \rbr*{ h^{-1} \br{ f( x+h, s) - f(x, s)}  }} \d s \\
        &= \int_E \rbr[\big]{\tfrac{\partial}{\partial x} f }(x, s ) \d s.
    \end{split}
\end{equation}
This completes the proof of \cref{lem:interchange}.
\end{proof}

\begin{cor} \label{cor:interchange}
Let $n \in \N$, $j \in \{1, 2, \ldots, n \}$, $\fu \in \R$, $\fv \in (\fu , \infty)$, let $f \colon \R^n \times [\fu , \fv] \to \R$ be locally Lipschitz continuous, let $F \colon \R^n \to \R$ satisfy for all $x \in \R^n$ that
\begin{equation}
    F(x) = \int_\fu ^\fv f (x , s ) \d s,
\end{equation}
let $x_1, x_2, \ldots, x_n \in \R$, let $E \subseteq [\fu , \fv ]$ be measurable, assume $\int_{[\fu , \fv] \backslash E } 1 \d s = 0 $, and assume for all $s \in E$ that $\R \ni v \mapsto f ( x_1, \ldots, x_{j-1}, v, x_{j+1}, \ldots, x_n, s) \in \R$ is differentiable at $x_j$. Then
\begin{enumerate} [(i)]
    \item \label{cor:interchange:item1} it holds that $\R \ni v \mapsto F(x_1, \ldots, x_{j-1}, v , x_{j+1}, \ldots, x_n) \in \R$ is differentiable at $x_j$ and
    \item \label{cor:interchange:item2} it holds that
    \begin{equation}
        \rbr[\big]{\tfrac{\partial}{\partial x_j} F} ( x_1, \ldots, x_n) = \int_E \rbr[\big]{\tfrac{\partial}{\partial x_j} f }(x_1, \ldots, x_n, s ) \d s.
    \end{equation}
\end{enumerate}
\end{cor}
\begin{proof} [Proof of \cref{cor:interchange}]
\Nobs that \cref{lem:interchange} establishes items \eqref{cor:interchange:item1} and \eqref{cor:interchange:item2}.
The proof of \cref{cor:interchange} is thus complete.
\end{proof}

\begin{lemma} \label{prop:loss:diff:vc}
Assume \cref{setting:const} and let $\phi = ( \phi_1, \ldots, \phi_{3 \width + 1 } ) \in \R^{3 \width + 1 }$. Then
\begin{enumerate} [(i)]
    \item \label{prop:loss:diff:vc:item1} it holds for all $j \in \N \cap (2 \width , 3 \width + 1 ]$ that $\R \ni v \mapsto \cL_\infty ( \phi_1, \ldots, \phi_{  j - 1}, v, \phi_{j+1}, \ldots, \phi_{3 \width + 1 }) \in \R$ is differentiable at $\phi_{ j}$ and
    \item \label{prop:loss:diff:vc:item2} it holds for all $j \in \N \cap (2 \width , 3 \width + 1 ]$ that $( \frac{\partial}{\partial \phi_{ j} } \cL_\infty )(\phi ) = \cG_{j } ( \phi)$.
\end{enumerate}
\end{lemma}
\begin{proof} [Proof of \cref{prop:loss:diff:vc}]
\Nobs that the fact that $\sigma_\infty$ is Lipschitz continuous assures that
\begin{equation}
    \R^{\width + 1 } \times [0,1] \ni (u_1, \ldots, u_{\width + 1 }, x ) \mapsto  \rbr[\big]{\realapprox{(\phi_1, \ldots, \phi_{2 \width}, u_1, \ldots,  u_{\width + 1 }) }{\infty} (x) - \alpha  }^2 \in \R
\end{equation}
is locally Lipschitz continuous. In addition, \nobs that for all $u_1, u_2, \ldots, u_{\width + 1 } \in \R$, $j \in \{1, 2, \ldots, \width + 1 \}$, $x \in [0,1]$ it holds that
\begin{equation}
   \R \ni v \mapsto \rbr[\big]{\realapprox{(\phi_1, \ldots, \phi_{2 \width}, u_1, \ldots, u_{j-1}, v, u_{j+1}, \ldots, u_{\width + 1 } ) }{\infty} (x) - \alpha  }^2 \in \R
\end{equation}
is differentiable at $u_j$. Moreover, \nobs that the chain rule implies that for all $j \in \{1, 2, \ldots, \width \}$, $x \in [0,1]$ it holds that
\begin{equation}
    \tfrac{\partial}{\partial \phi_{2 \width + j} } \br[\big]{ ( \realapprox{\phi}{\infty} ( x ) - \alpha ) ^2 } = 2  [\sigma_\infty (\phi_j x + \phi_{\width + j}) ] ( \realapprox{\phi}{\infty}(x) - \alpha )
\end{equation}
and
\begin{equation}
    \tfrac{\partial}{\partial \phi_{3 \width + 1} } \br[\big]{ ( \realapprox{\phi}{\infty} ( x ) - \alpha ) ^2 } = 2 ( \realapprox{\phi}{\infty}(x) - \alpha ).
\end{equation}
Combining this, \cref{cor:interchange}, and \eqref{eq:loss:gradient} establishes items \eqref{prop:loss:diff:vc:item1} and \eqref{prop:loss:diff:vc:item2}. The proof of \cref{prop:loss:diff:vc} is thus complete.
\end{proof}

\begin{lemma} \label{prop:loss:diff:wb}
Assume \cref{setting:const}, let $\phi = ( \phi_1, \ldots, \phi_{3 \width + 1 } ) \in \R^{3 \width + 1 }$, and let $j \in \{1, 2, \ldots, \width \}$, $i \in \{j , \width + j \}$ satisfy $| \phi_j | + | \phi_{\width + j } | > 0$. Then
\begin{enumerate} [(i)]
    \item \label{prop:loss:diff:wb:item1} it holds that $\R \ni v \mapsto \cL_\infty ( \phi_1, \ldots, \phi_{i-1}, v, \phi_{i+1}, \ldots, \phi_{3 \width + 1 }) \in \R$ is differentiable at $\phi_i$ and
    \item \label{prop:loss:diff:wb:item2} it holds that $( \frac{\partial}{\partial \phi_i } \cL_\infty )(\phi ) = \cG_{ i } ( \phi)$.
\end{enumerate}
\end{lemma}
\begin{proof} [Proof of \cref{prop:loss:diff:wb}]
Throughout this proof let $E \subseteq \R$ satisfy $E = \{ x \in [0,1] \colon \phi_j x + \phi_{\width + j } \not= 0 \}$. \Nobs that the assumption that $| \phi_j | + | \phi_{\width + j } | > 0$ implies that $\# ( [0,1] \backslash  E ) \leq 1$. This shows that $\int_{[0,1] \backslash E } 1 \d s = 0$.
Next \nobs that the fact that $\sigma_\infty$ is Lipschitz continuous ensures that
\begin{equation} \label{prop:loss:diff:wb:eq1}
    \R^{2 \width} \times [0,1] \ni (u_1, \ldots, u_{2 \width}, x) \mapsto \rbr[\big]{ \realapprox{(u_1, \ldots, u_{2\width } , \phi_{2 \width + 1 }, \ldots, \phi_{3 \width + 1 } )}{\infty} ( x ) - \alpha } ^2 \in \R
\end{equation} 
is locally Lipschitz continuous. In addition, \nobs that for all $x \in \R \backslash \{ 0 \}$ it holds that $\sigma_\infty$ is differentiable at $x$. Furthermore, \nobs that for all $x \in \R \backslash \{ 0 \}$ it holds that $(\sigma_\infty ) ' ( x ) = \indicator{(0, \infty)} ( x )$. This and the chain rule prove for all $x \in E$ that
\begin{equation} \label{prop:loss:diff:wb:eq2}
    \R \ni v \mapsto \rbr[\big]{ \realapprox{(\phi_1, \ldots, \phi_{j-1}, v, \phi_{j+1}, \ldots, \phi_{3 \width + 1 })}{\infty} ( x ) - \alpha}^2 \in \R
\end{equation}
is differentiable at $\phi_j$ and 
\begin{equation} \label{prop:loss:diff:wb:eq3}
    \tfrac{\partial }{\partial \phi_j} ( \realapprox{\phi}{\infty} ( x ) - \alpha ) ^2 = 2 \phi_{2 \width + j} x ( \realapprox{\phi}{\infty} ( x ) - \alpha ) \indicator{(0, \infty)} ( \phi_j x + \phi_{\width + j }) = 2 \phi_{2 \width + j} x ( \realapprox{\phi}{\infty} ( x ) - \alpha ) \indicator{I_j^\phi} ( x ).
\end{equation}
 Moreover, \nobs that the chain rule implies that for all $x \in E$ we have that
\begin{equation}
    \R \ni u \mapsto \rbr[\big]{ \realapprox{( \phi_1, \ldots, \phi_{\width + j -1}, u, \phi_{\width + j+1}, \ldots, \phi_{3 \width + 1 })}{\infty} ( x ) - \alpha}^2 \in \R
\end{equation}
is differentiable at $\phi_{\width + j}$ and
\begin{equation}
    \tfrac{\partial }{\partial \phi_{\width + j}} ( \realapprox{\phi}{\infty} ( x ) - \alpha ) ^2 = 2 \phi_{2 \width + j}  ( \realapprox{\phi}{\infty} ( x ) - \alpha ) \indicator{(0, \infty)} ( \phi_j x + \phi_{\width + j }) = 2 \phi_{2 \width + j}  ( \realapprox{\phi}{\infty} ( x ) - \alpha ) \indicator{I_j^\phi} ( x ).
\end{equation}
Combining \eqref{prop:loss:diff:wb:eq1}, \eqref{prop:loss:diff:wb:eq2}, \eqref{prop:loss:diff:wb:eq3}, \cref{cor:interchange}, and \eqref{eq:loss:gradient} hence establishes items \eqref{prop:loss:diff:wb:item1} and \eqref{prop:loss:diff:wb:item2}. The proof of \cref{prop:loss:diff:wb} is thus complete.
\end{proof}

\begin{lemma} \label{lem:loss:diff:degenerate}
Assume \cref{setting:const}, let $\phi = ( \phi_1, \ldots, \phi_{3 \width + 1 } ) \in \R^{3 \width + 1 }$, $j \in \{1, 2, \ldots, \width \}$, assume $\phi_j = \phi_{\width + j } = 0$, and assume that $\cL_\infty$ is differentiable at $\phi$. Then $(\frac{\partial}{\partial \phi_{ j }} \cL_\infty) ( \phi )  = \cG_{j } ( \phi ) = (\frac{\partial}{\partial \phi_{ \width + j }} \cL_\infty) ( \phi ) =  \cG_{ \width + j } ( \phi ) = 0$.
\end{lemma}
\begin{proof} [Proof of \cref{lem:loss:diff:degenerate}]
Throughout this proof let $\varphi^h  = (\varphi_1^h, \ldots, \varphi_{3 \width + 1 } ^h ) \in \R^{3 \width + 1 }$, $h = (h_1, h_2) \in \R^2 $, satisfy for all $h = (h_1, h_2) \in \R^2 $, $k \in \{1, 2, \ldots, 3 \width + 1 \} \backslash \{j , \width + j \}$ that $\varphi_j^h = \phi_j + h_1$, $\varphi_{\width+j}^h = \phi_{\width + j } + h_2$, and $ \varphi_k^h = \phi_k $. \Nobs that the assumption that $\cL_\infty$ is differentiable at $\phi$ ensures that for all $i \in \{j , \width + j \}$ it holds that $\R \ni v \mapsto \cL_\infty ( \phi_1, \ldots, \phi_{i-1}, v, \phi_{i+1}, \ldots, \phi_{3 \width + 1 }) \in \R$ is differentiable at $\phi_i$. Furthermore, \nobs that for all $h \in (- \infty , 0 ] ^2 $, $x \in [0,1]$ it holds that $\realapprox{\varphi^h}{\infty} ( x ) = \realapprox{\phi}{\infty} ( x ) $. Hence, we have for all $h \in (- \infty , 0 ] ^2 $ that $\cL_\infty ( \varphi^h ) = \cL_\infty ( \phi ) $. This implies that $(\frac{\partial}{\partial \phi_{ j }} \cL_\infty) ( \phi ) = (\frac{\partial}{\partial \phi_{ \width + j }} \cL_\infty) ( \phi ) = 0$. Moreover, \nobs that the assumption that $\phi_j = \phi_{\width + j } = 0$ implies that $I_j^\phi = \emptyset$. This and \eqref{eq:loss:gradient} demonstrate that $\cG_j ( \phi ) = \cG_{\width + j } ( \phi ) = 0$. Hence, we obtain that $(\frac{\partial}{\partial \phi_{ j }} \cL_\infty) ( \phi ) = 0 = \cG_{j } ( \phi )$ and $(\frac{\partial}{\partial \phi_{ \width + j }} \cL_\infty) ( \phi ) = 0 = \cG_{ \width + j } ( \phi )$. This completes the proof of \cref{lem:loss:diff:degenerate}.
\end{proof}

\begin{cor} \label{cor:loss:differentiable}
Assume \cref{setting:const}, let $\phi = ( \phi_1, \ldots, \phi_{3 \width + 1 } ) \in \R^{3 \width + 1 }$, and assume that $\cL_\infty$ is differentiable at $\phi$. Then $(\nabla \cL_\infty)(\phi) = \cG ( \phi ) $.
\end{cor}
\begin{proof} [Proof of \cref{cor:loss:differentiable}]
\Nobs that the assumption that $\cL_\infty$ is differentiable at $\phi$ ensures that for all $i \in \{1, 2, \ldots, 3 \width + 1 \}$ it holds that $\R \ni v \mapsto \cL_\infty ( \phi_1, \ldots, \phi_{i-1}, v, \phi_{i+1}, \ldots, \phi_{3 \width + 1 }) \in \R$ is differentiable at $\phi_i$.
Moreover, \nobs that \cref{prop:loss:diff:vc} proves for all $j \in \N \cap (2\width , 3 \width + 1 ] $ that $(\frac{\partial}{\partial \phi_{ j }} \cL_\infty) ( \phi ) = \cG_{ j } ( \phi )$.
In addition, \nobs that \cref{prop:loss:diff:wb} shows that for all $j \in \{1, 2, \ldots, \width \}$ with $| \phi_j | + | \phi_{\width + j } | > 0$ it holds that $(\frac{\partial}{\partial \phi_{ j }} \cL_\infty) ( \phi ) = \cG_{j } ( \phi )$ and $(\frac{\partial}{\partial \phi_{ \width + j }} \cL_\infty) ( \phi ) = \cG_{ \width + j } ( \phi )$. 
On the other hand, \nobs that \cref{lem:loss:diff:degenerate} ensures that for all $j \in \{1, 2, \ldots, \width \}$ with $\phi_j = \phi_{\width + j } = 0$ we have that $(\frac{\partial}{\partial \phi_{ j }} \cL_\infty) ( \phi ) = 0 = \cG_{j } ( \phi )$ and $(\frac{\partial}{\partial \phi_{ \width + j }} \cL_\infty) ( \phi ) = 0 = \cG_{ \width + j } ( \phi )$. This demonstrates that $(\nabla \cL_\infty ) ( \phi ) = \cG ( \phi )$.
The proof of \cref{cor:loss:differentiable} is thus complete.
\end{proof}

\subsection{Upper bounds for gradients of the risk functions}

\begin{lemma} \label{lem:gradient:est}
Assume \cref{setting:const} and let $\phi \in \R^{3 \width + 1}$. Then
\begin{equation} 
    \norm{ \cG ( \phi ) } ^2 \leq ( 8 \norm{ \phi } ^2 + 4 ) \cL_\infty ( \phi ).
\end{equation}
\end{lemma}
\begin{proof} [Proof of \cref{lem:gradient:est}]
Throughout this proof let $w_1,  \ldots, w_{\width}, b_1,  \ldots, b_{\width}, v_1,  \ldots, v_{\width}, c \in \R$ satisfy $\phi  = (w_1, \ldots, w_{\width}, b_1, \ldots, b_{\width}, v_1, \ldots, \allowbreak v_{\width}, c)$.
\Nobs that Jensen's inequality implies that
\begin{equation} \label{lem:grad:est:eq1}
    \rbr*{  \int_0^1 | \realapprox{\phi}{\infty} ( x ) - \alpha| \d x  } ^{\! 2} \leq \int_0^1 \rbr{  \realapprox{\phi}{\infty} ( x ) - \alpha } ^2 \d x = \cL_\infty ( \phi ).
\end{equation}
This and \eqref{eq:loss:gradient} ensure that for all $j \in \{1,2, \ldots, \width \}$ we have that
\begin{equation} \label{eq:lem:gradient:est1}
    \begin{split}
        | \cG_j( \phi) | ^2 &= 4 (v_j) ^2 \rbr*{ \int_{I_j^\phi} x ( \realapprox{\phi}{\infty} (x) - \alpha ) \d x  } ^{\! 2}
        \leq 4 (v_j) ^2 \rbr*{ \int_{I_j^\phi} |x|  | \realapprox{\phi}{\infty} (x) - \alpha  | \d x  } ^{ \! 2} \\
        &\leq 4 (v_j) ^2 \rbr*{ \int_0^1   | \realapprox{\phi}{\infty} (x) - \alpha  | \d x  } ^{\! 2} \leq 4 (v_j)^2 \cL_\infty (\phi).
    \end{split}
\end{equation}
In addition, \nobs that \eqref{eq:loss:gradient} and \eqref{lem:grad:est:eq1} assure that for all $j \in \{1,2, \ldots, \width \}$ it holds that
\begin{equation} \label{eq:lem:gradient:est2}
\begin{split}
         | \cG_{\width + j}( \phi) |^2 &= 4 (v_j) ^2 \rbr*{  \int_{I_j^\phi} (\realapprox{\phi}{\infty} (x) - \alpha  ) \d x  } ^{\! 2} \\
         &\leq 4 (v_j) ^2  \rbr*{ \int_0^1   |\realapprox{\phi}{\infty} (x) - \alpha  | \d x  } ^{\! 2} \leq 4 (v_j)^2 \cL_\infty (\phi).
         \end{split}
\end{equation}
Furthermore, \nobs that for all $x \in [0,1]$, $j \in \{1,2, \ldots, \width \}$ it holds that $| \sigma_\infty ( w_j  x + b_j ) | ^2 \leq ( | w_j | + | b_j | ) ^2 \leq 2 ( (w_j) ^2 + (b_j) ^2)$. Combining this and \eqref{eq:loss:gradient} demonstrates for all $j \in \{1,2, \ldots, \width \}$ that
\begin{equation} \label{eq:lem:gradient:est3}
\begin{split}
    | \cG_{2 \width + j} ( \phi ) | ^2 &= 4 \rbr*{  \int_0^1 [\sigma_\infty(w_j x + b_j )] ( \realapprox{\phi}{\infty}(x) - \alpha ) \d x  } ^{\! 2 } \\ 
    &\leq 4 \int_0^1 | \sigma_\infty ( w_j x + b_j ) | ^2 ( \realapprox{\phi}{\infty}(x) - \alpha ) ^2 \d x
    \leq 8 \br*{  (w_j) ^2 + (b_j) ^2 } \cL_\infty (\phi).
    \end{split}
\end{equation}
Finally, \nobs that \eqref{eq:loss:gradient} and \eqref{lem:grad:est:eq1} show that
\begin{equation} \label{eq:lem:gradient:est4}
    | \cG _ { 3 \width + 1 } ( \phi ) | ^2 = 4 \rbr*{ \int_0^1 (\realapprox{\phi}{\infty} (x) - \alpha  ) \d x  } ^{\! 2 } \leq 4 \cL_\infty ( \phi ).
\end{equation}
Combining \eqref{eq:lem:gradient:est1}--\eqref{eq:lem:gradient:est4} yields
\begin{equation}
\begin{split}
    \norm{ \cG ( \phi ) } ^2 &\leq  \br*{   \smallsum_{j=1}^\width \rbr*{  4 (v_j) ^2 + 4 (v_j) ^2 + 8 (w_j) ^2 + 8 (b_j) ^2  } } \cL_\infty ( \phi ) + 4 \cL_\infty ( \phi ) \\
    &\leq (8 \norm{ \phi } ^2 + 4) \cL_\infty ( \phi ).
    \end{split}
\end{equation}
The proof of \cref{lem:gradient:est} is thus complete.
\end{proof}

\begin{cor} \label{cor:g:bounded} 
Assume \cref{setting:const} and let $K \subseteq \R^{3 \width + 1}$ be a compact set. Then $\sup_{\phi \in K} \norm{ \cG ( \phi ) } < \infty$.
\end{cor}
\begin{proof} [Proof of \cref{cor:g:bounded}]
\Nobs that the fact that $\cL_\infty$ is continuous ensures that $\sup_{\phi \in K} \cL_\infty ( \phi ) < \infty$. Combining this with \cref{lem:gradient:est} completes the proof of \cref{cor:g:bounded}.
\end{proof}

\subsection{Properties of Lyapunov type functions} \label{subsection:lyapunov}

\begin{prop} \label{prop:lyapunov:norm} 
Assume \cref{setting:const} and let $\phi \in \R^{3 \width + 1}$. Then 
\begin{equation}
    \norm{ \phi } ^2 \leq V(\phi) \leq 3 \norm{ \phi } ^2 + 8 \alpha ^2.
\end{equation}
\end{prop}
\begin{proof} [Proof of \cref{prop:lyapunov:norm}]
\Nobs that 
$V(\phi) = \norm{ \phi } ^2 + ( \c{\phi} - 2 \alpha ) ^2 \geq \norm{ \phi } ^2$.
Furthermore, \nobs that the fact that $\forall \, x , y \in \R \colon (x - y )^2 \leq 2(x^2 + y^2)$ establishes that
\begin{equation}
    V (\phi) \leq \norm{ \phi } ^2 + 2 (\c{\phi}) ^2 + 8 \alpha ^2 \leq 3 \norm{ \phi } ^2 + 8 \alpha ^2.
\end{equation}
This completes the proof of \cref{prop:lyapunov:norm}.
\end{proof}

\begin{prop} \label{prop:v:gradient}
Assume \cref{setting:const} and let $\phi , \psi \in \R^{3 \width + 1}$. Then
\begin{equation} \label{eq:prop:v:gradient}
    (\nabla V)(\phi) - (\nabla V)(\psi) = 2(\phi - \psi ) + \rbr[\big]{ 0, 0, \ldots, 0, 2(\c{\phi} - \c{\psi}) }.
\end{equation}
\end{prop}
\begin{proof} [Proof of \cref{prop:v:gradient}]
\Nobs that for all $\varphi \in \R^{3 \width + 1}$ it holds that
\begin{equation}
    (\nabla V ) ( \varphi) = 2 \varphi + \rbr[\big]{  0, 0, \ldots, 0, 2(\c{\varphi} - 2 \alpha ) } .
\end{equation}
This establishes \eqref{eq:prop:v:gradient}. The proof of \cref{prop:v:gradient} is thus complete.
\end{proof}

\begin{prop} \label{prop:lyapunov:gradient} 
Assume \cref{setting:const}, let $\cV_1, \cV_2 \in  C ( \R^{3 \width + 1 } , \R )$ satisfy for all $\phi \in \R^{3 \width + 1}$ that
$\cV_1(\phi) = (\c{\phi}) ^2 - 2 \alpha \c{\phi} + \sum_{j=1}^\width (\v{\phi}_j ) ^2 $
and
$\cV_2(\phi) =  (\c{\phi})^2 - 2 \alpha \c{\phi} + \sum_{j=1}^\width \br[\big]{   (\w{\phi}_j)^2 + (\b{\phi}_j)^2 } $,
and let $\phi \in \R^{3 \width + 1}$. Then
\begin{enumerate} [(i)]
    \item \label{prop:lyapunov:gradient:item1} it holds that $\langle (\nabla \cV_1) ( \phi) , \cG(\phi) \rangle  = 4 \cL_\infty (\phi)$,
    \item \label{prop:lyapunov:gradient:item2} it holds that $\langle (\nabla \cV_2) ( \phi) , \cG(\phi) \rangle  = 4 \cL_\infty (\phi)$, and
    \item \label{prop:lyapunov:gradient:item3} it holds that $\langle (\nabla V) ( \phi) , \cG(\phi) \rangle = 8 \cL_\infty (\phi)$.
\end{enumerate} 
\end{prop}
\begin{proof} [Proof of \cref{prop:lyapunov:gradient}]
Throughout this proof let $w_1,  \ldots, w_{\width}, b_1,  \ldots, b_{\width}, v_1,  \ldots, v_{\width}, c \in \R$ satisfy $\phi  = (w_1, \ldots, w_{\width}, b_1, \ldots, b_{\width}, v_1, \ldots, \allowbreak v_{\width}, c)$.
\Nobs that
\begin{equation}
    (\nabla \cV_1) ( \phi ) = 2 \rbr[\big]{  \underbrace{0,0, \ldots, 0}_{2 \width}, v_1, v_2, \ldots, v_{\width},  c - \alpha } .
\end{equation} 
This and \eqref{eq:loss:gradient} imply that
\begin{equation}
    \begin{split}
        &\langle (\nabla \cV_1) ( \phi) , \cG(\phi) \rangle \\ 
        &= 4 \br[\Bigg]{ \sum_{j=1}^\width v_j \int_0^1 [\sigma_\infty (w_j x + b_j) ] ( \realapprox{\phi}{\infty}(x) - \alpha) \d x } + 4(c - \alpha) \int_0^1 (\realapprox{\phi}{\infty} (x) - \alpha ) \d x \\
        & = 4 \int_0^1 \rbr*{ \br*{ \smallsum_{j=1}^\width v_j \sigma_\infty ( w_j x + b_j) } + c - \alpha  } (\realapprox{\phi}{\infty} (x) - \alpha ) \d x \\
        &= 4 \int_0^1 (\realapprox{\phi}{\infty} (x) - \alpha )^2 \d x = 4 \cL_\infty ( \phi ).
    \end{split}
\end{equation}
This proves item \eqref{prop:lyapunov:gradient:item1}.
Next \nobs that
\begin{equation}
    (\nabla \cV_2 ) ( \phi ) = 2 \rbr[\big]{  w_1, w_2, \ldots, w_{\width}, b_1, b_2, \ldots, b_{\width}, \underbrace{0, 0, \ldots, 0}_{\width }, c -\alpha }.
\end{equation}
Combining this and \eqref{eq:loss:gradient} demonstrates that
\begin{equation}
    \begin{split}
        &\langle (\nabla \cV_2) ( \phi) , \cG(\phi) \rangle \\
        &= 4 \br[\Bigg]{  \sum_{j=1}^\width v_j \int_{I_j^\phi} (w_j x + b_j) (\realapprox{\phi}{\infty} (x) - \alpha ) \d x } + 4(c - \alpha) \int_0^1 (\realapprox{\phi}{\infty} (x) - \alpha ) \d x \\
        &= 4 \br[\Bigg]{ \sum_{j=1}^\width v_j \int_0^1 \br{\sigma_\infty(w_j x + b_j) } ( \realapprox{\phi}{\infty}(x) - \alpha) \d x} + 4(c - \alpha) \int_0^1 \rbr{\realapprox{\phi}{\infty} (x) - \alpha } \d x \\
        & = 4 \int_0^1 \rbr*{ \br*{ \smallsum_{j=1}^\width v_j \sigma_\infty ( w_j  x + b_j) } + c - \alpha  } (\realapprox{\phi}{\infty} (x) - \alpha ) \d x \\
        &= 4 \int_0^1 (\realapprox{\phi}{\infty} (x) - \alpha )^2 \d x = 4 \cL_\infty ( \phi ).
    \end{split}
\end{equation}
This establishes item \eqref{prop:lyapunov:gradient:item2}.
Furthermore, \nobs that $\cV_1 ( \phi ) + \cV _2 ( \phi ) = V ( \phi ) - 4 \alpha ^2$. This ensures that $(\nabla \cV_1 ) ( \phi ) + ( \nabla \cV_2 ) ( \phi ) = ( \nabla V ) ( \phi)$, which proves item \eqref{prop:lyapunov:gradient:item3}.
The proof of \cref{prop:lyapunov:gradient} is thus complete.
\end{proof}

\begin{cor} \label{cor:critical:points} 
Assume \cref{setting:const} and let $\phi \in \R^{3 \width + 1}$. Then it holds that $ \norm{ \cG(\phi) } = 0$ if and only if $\cL_\infty ( \phi ) = 0 $.
\end{cor}

\begin{proof} [Proof of \cref{cor:critical:points}]
Assume first that $\norm{ \cG( \phi ) } = 0$.
Then \cref{prop:lyapunov:gradient} implies that $8 \cL_\infty (\phi) = \langle (\nabla V ) ( \phi) , \cG(\phi) \rangle = 0$.
Next assume $\cL_\infty (\phi) = \int_0^1 (\realapprox{\phi}{\infty} (x) - \alpha)^2 \d x = 0$.
The fact that $\realapprox{\phi}{\infty} \in C(\R , \R)$ then implies that it holds for all $x \in [0,1]$ that $\realapprox{\phi}{\infty}(x) = \alpha$.
Hence, \eqref{eq:loss:gradient} demonstrates that $\cG(\phi) = 0 \in \R^{3 \width + 1 }$ and therefore $\norm{ \cG ( \phi ) } = 0$.
This completes the proof of \cref{cor:critical:points}.
\end{proof}

\section{Convergence analysis for gradient flow processes}
\label{section:gradientflow}
In this section we employ the findings from \cref{section:risk:regularity} to establish in \cref{theo:flow} below that the risks of the considered time-continuous gradient flow processes converge to zero. Our proof of \cref{theo:flow} uses the deterministic It\^{o} type formula for the Lyapunov function $V \colon \R^{ 3 \width + 1 } \to \R$ from \cref{setting:const}, which we establish in \cref{lem:flow:lyapunov} in \cref{subsection:ito:lyapunov} below, as well as the deterministic It\^{o} type formula for the risk function $\cL_{ \infty } \colon \R^{ 3 \width + 1 } \to \R$ from \cref{setting:const}, which we establish in \cref{lem:loss:integral} in \cref{subsection:ito:risk} below.

Our proof of the deterministic It\^{o} type formula for the Lyapunov function $V \colon \R^{ 3 \width + 1 } \to \R$ in \cref{lem:flow:lyapunov}, in turn, is based on the fact that the function $V \colon \R^{ 3 \width + 1 } \to \R$ from \cref{setting:const} satisfies the Lyapunov property in item \eqref{prop:lyapunov:gradient:item3} in \cref{prop:lyapunov:gradient} as well as on the well-known deterministic It\^{o}-type formula for continuously differentiable functions in \cref{lem:chainrule:gen} in \cref{subsection:ito:lyapunov} below. We include in this section a detailed proof for \cref{lem:chainrule:gen} only for completeness.

In contrast to \cref{lem:flow:lyapunov}, 
the deterministic It\^{o} type formula 
for the risk function $\cL_{ \infty } \colon \R^{ 3 \width + 1 } \allowbreak \to \R$
in \cref{lem:loss:integral} can not be proved through 
an application of \cref{lem:chainrule:gen} as the risk function 
$\cL_{ \infty } \colon \R^{ 3 \width + 1 } \to \R$ fails to be differentiable.
Instead we prove \cref{lem:loss:integral} through an approximation 
argument by employing the mollified 
rectifier functions $\sigma_r \in C^{ \infty }( \R, \R )$, $r \in [1,\infty)$, and their corresponding 
risk functions $\cL_r \colon \R^{ 3 \width + 1  } \to \R$, $r \in [1,\infty)$,
from \cref{setting:const}.

\subsection{Deterministic It\^{o} formulas for Lyapunov type functions}
\label{subsection:ito:lyapunov}

\begin{lemma} \label{lem:chainrule:gen}
Let $T \in (0, \infty)$, $n \in \N$, $\Theta \in C ( [ 0, T ] , \R^n )$, $F \in C^1 ( \R^n, \R)$, let $\vartheta \colon [0, T] \to \R^n$ be a bounded measurable function, and assume for all $t \in [0, T]$ that
\begin{equation}
    \Theta_t = \Theta_0 + \int_0^t \vartheta_s \d s.
\end{equation}
Then it holds for all $t \in [0, T]$ that
\begin{equation} 
    F ( \Theta_t ) = F ( \Theta_0 ) + \int_0^t \rbr[\big]{  F' ( \Theta_s ) }  \vartheta_s \d s.
\end{equation}
\end{lemma}
\begin{proof} [Proof of \cref{lem:chainrule:gen}]
\Nobs that the fact that $\vartheta$ is bounded proves that $\Theta$ is Lipschitz continuous.  Combining this and Rademacher's theorem shows that there exists a measurable set $E \subseteq [0,T]$ which satisfies that $\int_{[0, T] \backslash E} 1 \d s = 0 $, which satisfies for all $t \in E$ that $[0, T ] \ni s \mapsto \Theta_s \in \R^n$ is differentiable at $t$, and which satisfies for all $t \in E$ that $\frac{\d}{\d t} \Theta_t = \vartheta_t$. This and the chain rule demonstrate that for all $t \in E$ it holds that $[0, T ] \ni s \mapsto F ( \Theta_s ) \in \R$ is differentiable at $t$ and that $\frac{\d}{\d t} ( F ( \Theta_t ) ) = ( F' ( \Theta_t ) )  \vartheta_t$. Furthermore, \nobs that the fact that $\Theta$ is Lipschitz continuous and the fact that $F$ is continuously differentiable establish that $[0, T ] \ni t \mapsto F ( \Theta_t ) \in \R$ is Lipschitz continuous. Hence, we obtain that $[0, T ] \ni t \mapsto F ( \Theta_t ) \in \R$ is absolutely continuous. This shows for all $t \in [0,T]$ that
\begin{equation}
    F ( \Theta_t ) = F ( \Theta_0 ) + \int_0^t \rbr[\big]{  F' ( \Theta_s ) }  \vartheta_s \d s.
\end{equation}
The proof of \cref{lem:chainrule:gen} is thus complete.
\end{proof}

\begin{lemma} \label{lem:flow:lyapunov}
Assume \cref{setting:const}, let $T \in (0, \infty)$, and let $\Theta \in C([0, T] , \R^{3 \width + 1})$ satisfy for all $t \in [0, T]$ that $\Theta_t = \Theta_0 - \int_0^t \cG ( \Theta_s ) \d s$. 
Then it holds for all $t \in [0, T]$ that $V(\Theta_t) = V(\Theta_0) - 8 \int_0^t \cL_\infty (\Theta_s) \d s$.
\end{lemma}
\begin{proof} [Proof of \cref{lem:flow:lyapunov}] 
\Nobs that \cref{cor:g:bounded} and the assumption that $\Theta \in C([0,T] , \R^{3 \width + 1 } )$ imply that $[0, T ] \ni t \mapsto \cG ( \Theta_t ) \in \R^{3 \width + 1}$ is bounded.
Combining this, the fact that $V \in C^\infty ( \R^{3 \width + 1} , \R )$, \cref{prop:limit:lr}, \cref{lem:chainrule:gen}, and \cref{prop:lyapunov:gradient} demonstrates that for all $t \in [0,T]$ we have that
\begin{equation}
    V(\Theta_t) - V(\Theta_0) = - \int_0^t  \langle ( \nabla V ) (\Theta_s), \cG(\Theta_s) \rangle \d s =- 8 \int_0^t \cL_\infty (\Theta_s) \d s.
\end{equation}
The proof of \cref{lem:flow:lyapunov} is thus complete.
\end{proof}

\begin{cor} \label{cor:flow:stability} 
Assume \cref{setting:const} and let $\Theta \in C([0, \infty) , \R^{3 \width + 1})$ satisfy for all $t \in [0, \infty)$ that $\Theta_t = \Theta_0 - \int_0^t \cG ( \Theta_s ) \d s$. 
Then $\sup_{t \in [ 0, \infty)} \norm{ \Theta_t } \leq \br{ V ( \Theta_0 ) } ^{1/2} < \infty$.
\end{cor}
\begin{proof} [Proof of \cref{cor:flow:stability}]
\Nobs that \cref{prop:lyapunov:norm} implies for all $t \in [ 0, \infty)$ that $\norm{ \Theta_t } \leq \br{ V(\Theta_t) }^{1/2} $.
Furthermore, \nobs that \cref{lem:flow:lyapunov} and the fact that $\forall \, \phi \in \R^{3 \width + 1} \colon \cL_\infty ( \phi) \geq 0$ demonstrate for all $t \in [ 0, \infty)$ that $V(\Theta_t) \leq V(\Theta_0)$. This completes the proof of \cref{cor:flow:stability}.
\end{proof}

\subsection{Deterministic It\^{o} formulas for risk functions}
\label{subsection:ito:risk}

\begin{lemma} \label{lem:lr:bounded} 
Assume \cref{setting:const} and let $K \subseteq \R^{3 \width + 1}$ be a compact set. Then $\sup_{\phi \in K} \allowbreak \sup_{r \in [1 , \infty) } \allowbreak \norm{ ( \nabla \cL_r ) ( \phi ) } < \infty$.
\end{lemma}
\begin{proof} [Proof of \cref{lem:lr:bounded}]
\Nobs that \cref{prop:relu:approximation} demonstrates for all $r \in [1 , \infty)$, $\phi = (w_1, \ldots, \allowbreak w_{\width}, b_1, \ldots, b_{\width}, v_1, \ldots, \allowbreak v_{\width}, c) \in \R^{3 \width + 1 }$, $x \in [0,1]$ that
\begin{equation}
   \abs{ \realapprox{\phi}{r} ( x) }
   \leq | c | + \smallsum_{j=1}^\width | v_j | ( \sigma_\infty ( w_j x + b_j ) + 1 ) 
   \leq | c | + \smallsum_{j=1}^\width | v_j | ( | w_j | + | b_j | + 1 ).
\end{equation}
Hence, we obtain for all $r \in [1, \infty)$, $\phi = (w_1, \ldots, w_{\width}, b_1, \ldots, b_{\width}, v_1, \ldots, \allowbreak v_{\width}, c) \in \R^{3 \width + 1 }$ that
\begin{equation}
    \cL_r ( \phi )
    \leq \int_0^1 \rbr[\big]{ | \alpha | +  | \realapprox{\phi}{r} ( x)| } ^2 \d x
    \leq \rbr*{ | \alpha | + | c | + \smallsum_{j=1}^\width | v_j | ( | w_j | + | b_j | + 1 ) }^2.
\end{equation}
This implies that $\sup_{\phi \in K} \sup_{r \in [1 , \infty)} \cL_r ( \phi ) < \infty$. Next \nobs that \eqref{eq:approx:loss:gradient} and the Cauchy-Schwarz inequality demonstrate that for all $r \in [1, \infty)$, $\phi = (w_1, \ldots, w_{\width}, b_1, \ldots, b_{\width}, v_1, \ldots, \allowbreak v_{\width}, c) \in \R^{3 \width + 1 }$ it holds that
\begin{equation} \label{eq:lr:bounded:1}
    \abs*{   \rbr*{  \tfrac{\partial }{ \partial c}  \cL_r  } ( \phi ) } 
    \leq 2 \int_0^1 | \realapprox{\phi}{r}(x) - \alpha | \d x \leq 2 \sqrt{\cL_r ( \phi)}.
\end{equation}
Furthermore, \nobs that the Cauchy-Schwarz inequality, \cref{prop:relu:approximation}, and \eqref{eq:approx:loss:gradient} prove that for all $r \in [1, \infty)$, $\phi = (w_1, \ldots, w_{\width}, b_1, \ldots, b_{\width}, v_1, \ldots, \allowbreak v_{\width}, c) \in \R^{3 \width + 1 }$, $j \in \{1, 2, \ldots, \width \}$ it holds that
\begin{equation} \label{eq:lr:bounded:2}
\begin{split}
    \abs[\big]{  \rbr[\big]{  \tfrac{\partial } { \partial w_j } \cL_r } ( \phi ) }
    &\leq 2 | v_j | \int_0^1 |x (\sigma_r) ' ( w_j x + b_j )|  | \realapprox{\phi}{r}(x) - \alpha | \d x \\
    &\leq 2 | v_j | \int_0^1 | \realapprox{\phi}{r}(x) - \alpha | \d x
    \leq 2 | v_j | \sqrt{\cL_r ( \phi ) }
    \end{split}
\end{equation}
and
\begin{equation} \label{eq:lr:bounded:3}
     \begin{split}
     \abs[\big]{  \rbr[\big]{  \tfrac{\partial } { \partial b_j } \cL_r } ( \phi ) }
     &\leq 2 | v_j | \int_0^1 | (\sigma_r) ' ( w_j x + b_j )|   | \realapprox{\phi}{r}(x) - \alpha | \d x \\
    & \leq 2 | v_j | \int_0^1 | \realapprox{\phi}{r}(x) - \alpha | \d x 
     \leq 2 | v_j | \sqrt{\cL_r ( \phi ) } .
     \end{split}
\end{equation}
In addition, \nobs that the Cauchy-Schwarz inequality, \cref{prop:relu:approximation}, and \eqref{eq:approx:loss:gradient} demonstrate that for all $r \in [1, \infty)$, $\phi = (w_1, \ldots, w_{\width}, b_1, \ldots, b_{\width}, v_1, \ldots, \allowbreak v_{\width}, c) \in \R^{3 \width + 1 }$, $j \in \{1, 2, \ldots, \width \}$ it holds that
\begin{equation}
\begin{split} 
    \abs[\big]{  \rbr[\big]{   \tfrac{\partial }{ \partial v_j} \cL_r } ( \phi ) } 
    & \leq  2  \int_0^1 [ \sigma_r ( w_j x + b_j) ] | \realapprox{\phi}{r}(x) - \alpha | \d x \\
    &\leq 2 ( 1 + | w_j | + | b_j | )\int_0^1 | \realapprox{\phi}{r}(x) - \alpha | \d x \\ &
    \leq 2 ( 1 + | w_j | + | b_j | ) \sqrt{ \cL_r ( \phi ) }.
\end{split}
\end{equation}
This, \eqref{eq:lr:bounded:1}, \eqref{eq:lr:bounded:2}, and \eqref{eq:lr:bounded:3} show that for all $r \in [1, \infty)$, $\phi = (w_1, \ldots, w_{\width}, b_1, \ldots, b_{\width}, v_1, \ldots, \allowbreak v_{\width}, c) \in \R^{3 \width + 1 }$ it holds that
\begin{equation}
    \norm{ (\nabla \cL_r) ( \phi ) } ^2 \leq \br*{  4 + \smallsum_{j=1}^\width \rbr*{  8 (v_j) ^2 + 4 ( 1 + | w_ j| + | b_j | ) ^2  } } \cL_r ( \phi)  .
\end{equation}
Combining this with the fact that $\sup_{\phi \in K} \sup_{r \in [1 , \infty) } \cL_r ( \phi ) < \infty$ establishes that
\begin{equation}
    \sup\nolimits_{\phi \in K} \sup\nolimits_{ r \in [1 , \infty) } \norm{ (\nabla \cL_r) ( \phi ) } ^2 < \infty.
\end{equation}
The proof of \cref{lem:lr:bounded} is thus complete.
\end{proof}

\begin{lemma} \label{lem:loss:integral} 
Assume \cref{setting:const}, let $T \in (0, \infty)$, and let $\Theta \in C([0, T ] , \R^{3 \width + 1} )$ satisfy for all $t \in [0,T] $ that $\Theta_t = \Theta_0 - \int_0^t \cG ( \Theta_s ) \d s$.
Then it holds for all $t \in [0,T]$ that $\cL_\infty (\Theta_t) = \cL_\infty (\Theta_0) - \int_0^t \norm{ \cG( \Theta_s ) } ^2 \d s$.
\end{lemma}
\begin{proof} [Proof of \cref{lem:loss:integral}]
\Nobs that \cref{lem:chainrule:gen} and item \eqref{prop:limit:lr:1} in \cref{prop:limit:lr} demonstrate that for all $r \in [1 , \infty)$, $t \in [0,T]$ it holds that
\begin{equation} \label{eq:lem:loss:integral}
    \cL_r ( \Theta_t) - \cL_r ( \Theta_0) = - \int_0^t \langle (\nabla \cL_r) ( \Theta_s), \cG ( \Theta_s ) \rangle \d s.
\end{equation}
Next \nobs that \cref{prop:limit:lr} proves that for all $t \in [0,T]$ it holds that $\lim_{r \to \infty} (  \cL_r ( \Theta_t) - \cL_r ( \Theta_0)) =   \cL_\infty ( \Theta_t) - \cL  ( \Theta_0)$.
Furthermore, \nobs that \cref{prop:limit:lr} ensures that for all $s \in [0,T]$ we have that $\lim_{r \to \infty} \langle ( \nabla \cL_r ) ( \Theta_s), \cG ( \Theta_s ) \rangle = \langle \cG ( \Theta_s), \cG ( \Theta_s ) \rangle =  \norm{  \cG ( \Theta_s ) } ^2$.
In addition, \nobs that the assumption that $\Theta \in C([0,T], \R^{3 \width + 1 })$ implies that there exists a compact set $K \subseteq \R^{3 \width + 1}$ such that $\forall \, s \in [0, T] \colon \Theta_s \in K$.
Combining this, the Cauchy-Schwarz inequality, \cref{cor:g:bounded}, and \cref{lem:lr:bounded} shows that
\begin{equation}
\begin{split}
    &\sup\nolimits_{r \in [1 , \infty)} \sup\nolimits_{s \in [0,T]} | \langle ( \nabla \cL_r) ( \Theta_s), \cG ( \Theta_s ) \rangle | \\ 
    &\leq \sup\nolimits_{r \in [1 , \infty)} \sup\nolimits_{\phi \in K}  | \langle (\nabla \cL_r ) ( \phi), \cG ( \phi ) \rangle | \\
    &\leq \sup\nolimits_{r \in [1 , \infty)} \sup\nolimits_{\phi \in K}  \rbr[\big]{  \norm{ (\nabla \cL_r) ( \phi ) }  \norm{  \cG ( \phi ) } } < \infty.
    \end{split}
\end{equation}
The dominated convergence theorem hence proves that for all $t \in [0,T]$ we have that
\begin{equation}
    \lim_{r \to \infty} \br*{  \int_0^t \langle (\nabla \cL_r) ( \Theta_s), \cG ( \Theta_s ) \rangle \d s }
    = \int_0^t  \br*{  \lim_{r \to \infty} \langle (\nabla \cL_r) ( \Theta_s), \cG ( \Theta_s ) \rangle } \d s
    = \int_0^t \norm{ \cG ( \Theta_s ) } ^2 \d s.
\end{equation}
Combining this with \eqref{eq:lem:loss:integral} completes the proof of \cref{lem:loss:integral}.
\end{proof}

\subsection{Convergence of the risks of gradient flow processes}

\begin{lemma} \label{lem:loss:decreasing}
Assume \cref{setting:const} and let $\Theta \in C([0, \infty) , \R^{3 \width + 1})$ satisfy for all $t \in [0, \infty)$ that $\Theta_t = \Theta_0 - \int_0^t \cG ( \Theta_s ) \d s$. Then it holds that $[0, \infty) \ni t \mapsto \cL_\infty ( \Theta_t) \in [0, \infty)$ is non-increasing.
\end{lemma}
\begin{proof} [Proof of \cref{lem:loss:decreasing}]
This is an immediate consequence of \cref{lem:loss:integral}.
\end{proof}

\begin{theorem} \label{theo:flow}
Assume \cref{setting:const} and let $\Theta \in C([0, \infty) , \R^{3 \width + 1})$ satisfy for all $t \in [0, \infty)$ that $\Theta_t = \Theta_0 - \int_0^t \cG ( \Theta_s ) \d s$. Then 
\begin{enumerate} [(i)] 
    \item \label{theo:flow:item1} it holds that $\sup_{t \in [0, \infty)} \norm{ \Theta_t } \leq \br{ V(\Theta_0 )} ^{1/2} < \infty$,
    \item \label{theo:flow:item2} it holds for all $t \in (0, \infty)$ that $\cL_\infty ( \Theta_t ) \leq \frac{V ( \Theta_0 ) }{8 t}$, and
    \item \label{theo:flow:item3} it holds that $\limsup_{t \to \infty} \cL_\infty (\Theta_t) = 0$.
\end{enumerate}
\end{theorem}
\begin{proof} [Proof of \cref{theo:flow}]
\Nobs that \cref{cor:flow:stability} establishes item \eqref{theo:flow:item1}. Next \nobs that \cref{lem:flow:lyapunov} and \cref{lem:loss:decreasing} prove that for all $t \in [0, \infty)$ it holds that
\begin{equation} 
    t \cL_\infty ( \Theta_t ) = \int_0^t \cL_\infty ( \Theta_t ) \d s \leq \int_0^t \cL_\infty ( \Theta_s) \d s = \frac{V(\Theta_0) - V ( \Theta_t) }{8} \leq \frac{ V( \Theta_0)}{8} < \infty.
\end{equation}
Hence, we obtain for all $t \in (0, \infty)$ that
\begin{equation}
   \cL_\infty ( \Theta_t ) \leq \frac{V ( \Theta_0 ) }{8 t}.
\end{equation}
This establishes items \eqref{theo:flow:item2} and \eqref{theo:flow:item3}. The proof of \cref{theo:flow} is thus complete.
\end{proof}

\section{Convergence analysis for gradient descent processes}
\label{section:gradientdescent}
In this section we use the findings from \cref{section:risk:regularity} to prove in \cref{theo:gd:loss} in \cref{subsection:theorem:gd} below that the risks of the considered time-discrete GD processes converge to zero. Our proof of \cref{theo:gd:loss} uses the fact that the function $V \colon \R^{ 3 \width + 1 } \to \R$ from \cref{setting:const} is also a Lyapunov function for the considered time-discrete GD processes, which we establish in \cref{lem:loss:decreasing} below. Moreover, in \cref{subsection:gd:random:initialization} below we apply \cref{theo:gd:loss} to establish in \cref{cor:gd:random} that also the expectations of risks of the time-discrete GD processes with random initializations do converge to zero.

\subsection{Lyapunov type estimates for gradient descent processes}
\label{subsection:gd:lyapunov}

\begin{lemma} \label{lem:est:vtheta_n}
Assume \cref{setting:const}, let $\gamma \in (0, \infty)$, and let $\Theta = (\Theta_n)_{n \in \N_0} = ( ( \Theta_n^1,  \ldots, \allowbreak \Theta_n^{3 \width + 1}))_{n \in \N_0} \colon \allowbreak \N_0 \to  \R^{3 \width + 1}$ satisfy for all $n \in \N_0$ that $\Theta_{n+1} = \Theta_n - \gamma \cG ( \Theta_n)$. Then it holds for all $n \in \N_0$ that
\begin{equation}  
    V(\Theta_{n+1}) -  V ( \Theta_n) \leq - 8 \gamma \cL_\infty (\Theta_n) + 2\gamma ^2 \norm{ \cG ( \Theta_n ) } ^2.
\end{equation}
\end{lemma}
\begin{proof} [Proof of \cref{lem:est:vtheta_n}]
Throughout this proof let $n \in \N_0$ be arbitrary and let $g \colon \R \to \R$ satisfy for all $t \in \R$ that $g(t) = V ( t \Theta_{n+1} + ( 1-t) \Theta_n )$. The fact that $V$ is continuously differentiable establishes that $g$ is continuously differentiable. The fundamental theorem of calculus and the chain rule hence ensure that
\begin{equation}
    \begin{split}
        &V(\Theta_{n+1}) - V(\Theta_n) =g(1)-g(0) = \int_0^1 g'(t) \d t\\
        &= \int_0^1 \langle (\nabla V) (t \Theta_{n+1} + (1-t) \Theta_n), \Theta_{n+1} - \Theta_n \rangle \d t \\
        &= - \gamma \int_0^1 \langle (\nabla V) (t \Theta_{n+1} + (1-t) \Theta_n), \cG ( \Theta_n ) \rangle \d t \\
        &= - \gamma \int_0^1 \langle (\nabla V ) (\Theta_n) , \cG(\Theta_n) \rangle \d t \\
        &\quad - \gamma \int_0^1 \langle (\nabla V) (t \Theta_{n+1} + (1-t) \Theta_n) - (\nabla V) ( \Theta_n), \cG ( \Theta_n ) \rangle \d t.
    \end{split}
\end{equation}
Next \nobs that \cref{prop:lyapunov:gradient} implies that $\langle (\nabla V ) (\Theta_n) , \cG(\Theta_n) \rangle = 8 \cL_\infty (\Theta_n)$. Furthermore, \nobs that \cref{prop:v:gradient} establishes for all $t \in [0, 1]$ that
\begin{equation}
    \begin{split}
        &\langle (\nabla V) (t \Theta_{n+1} + (1-t) \Theta_n) - (\nabla V ) ( \Theta_n), \cG ( \Theta_n ) \rangle \\
        &= \langle (\nabla V) (t (\Theta_{n+1} - \Theta_n) + \Theta_n ) - (\nabla V ) ( \Theta_n), \cG ( \Theta_n ) \rangle \\
        &= 2t \langle \Theta_{n+1} - \Theta_n, \cG(\Theta_n) \rangle + 2t (\Theta_{n+1}^{3 \width + 1} - \Theta_n^{3 \width + 1})\cG_{3 \width +1}(\Theta_n) \\
        &= -2 t \gamma \norm{  \cG ( \Theta_n ) } ^2 - 2 t \gamma |\cG_{3 \width +1}(\Theta_n) | ^2 \geq - 4 t \gamma \norm{  \cG ( \Theta_n ) } ^2.
    \end{split}
\end{equation}
Hence, we obtain that
\begin{equation}
    \begin{split}
        V(\Theta_{n+1}) - V(\Theta_n) 
        &\leq - 8 \gamma \cL_\infty (\Theta_n) + 4 \gamma ^2 \int_0^1 t \norm{ \cG ( \Theta_n ) } ^2 \d t \\
        &= - 8 \gamma \cL_\infty (\Theta_n) + 2\gamma ^2 \norm{ \cG ( \Theta_n ) } ^2.
    \end{split}
\end{equation}
The proof of \cref{lem:est:vtheta_n} is thus complete.
\end{proof}

\begin{cor} \label{cor:est:vtheta_n}
Assume \cref{setting:const}, let $\gamma \in (0, \infty)$, and let $\Theta = (\Theta_n)_{n \in \N_0}  \colon \N_0 \to \R^{3 \width + 1}$ satisfy for all $n \in \N_0$ that $\Theta_{n+1} = \Theta_n - \gamma \cG ( \Theta_n)$. Then it holds for all $n \in \N_0$ that
\begin{equation} 
      V(\Theta_{n+1}) -  V ( \Theta_n) \leq 8  \rbr*{  -  \gamma + \gamma ^2 ( 2 V(\Theta_n) + 1)  } \cL_\infty ( \Theta _ n) .
\end{equation}
\end{cor}
\begin{proof} [Proof of \cref{cor:est:vtheta_n}]
\Nobs that \cref{lem:gradient:est} and \cref{prop:lyapunov:norm} imply for all $n \in \N_0$ that
\begin{equation} 
\begin{split}
    \norm{ \cG ( \Theta_n ) } ^2 
    &\leq ( 8 \norm{ \Theta_n } ^2 + 4) \cL_\infty (\Theta_n)
    = 4 ( 2 \norm{ \Theta_n } ^2 + 1 )  \cL_\infty ( \Theta_n ) \\
    &\leq 4  (2 V ( \Theta_n) + 1 ) \cL_\infty (\Theta_n).
    \end{split}
\end{equation}
Combining this and \cref{lem:est:vtheta_n} ensures that for all $n \in \N_0$ we have that
\begin{equation}
\begin{split}
      V(\Theta_{n+1}) -  V ( \Theta_n) &\leq - 8 \gamma \cL_\infty ( \Theta_n ) + 8 \gamma^2  ( 2 V ( \Theta_n) + 1 ) \cL_\infty ( \Theta_n ) \\
      &= 8  \rbr*{  -  \gamma + \gamma ^2 ( 2 V(\Theta_n) + 1)  } \cL_\infty ( \Theta _ n).
      \end{split}
\end{equation}
The proof of \cref{cor:est:vtheta_n} is thus complete.
\end{proof}

\begin{lemma} \label{lem:vthetan:decreasing} 
Assume \cref{setting:const}, let $\gamma \in (0, \infty)$,  and let $\Theta = (\Theta_n)_{n \in \N_0}  \colon \N_0 \to \R^{3 \width + 1}$ satisfy for all $n \in \N_0$ that $\Theta_{n+1} = \Theta_n - \gamma \cG ( \Theta_n)$ and  $\gamma \leq (4 V ( \Theta_0) + 2 )^{-1}$. Then it holds  for all $n \in \N_0$ that $V (\Theta_{n+1}) - V ( \Theta_n) \leq - 4 \gamma \cL_\infty (\Theta_n) \leq 0$.
\end{lemma}
\begin{proof} [Proof of \cref{lem:vthetan:decreasing}]
We prove the statement by induction on $n \in \N_0$. \Nobs that \cref{cor:est:vtheta_n} implies that
\begin{equation}
\begin{split}
    V(\Theta_1) - V(\Theta_0) &\leq \rbr*{  - 8 \gamma + 8\gamma ^2 ( 2 V(\Theta_0) + 1 )  }  \cL_\infty ( \Theta _ 0) \\
    &\leq  \rbr*{  - 8 \gamma + 8 \gamma \br*{ \tfrac{2 V(\Theta_0) + 1}{4 V(\Theta_0) + 2} }  } \cL_\infty ( \Theta _ 0)  = - 4 \gamma \cL_\infty (\Theta_0) \leq 0.
    \end{split}
\end{equation}
This establishes the assertion in the base case $n=0$. For the induction step let $n \in \N$ satisfy for all $m \in \{0, 1, \ldots, n-1\}$ that 
\begin{equation} \label{eq:induction:1}
V( \Theta_{m + 1}) - V ( \Theta_{m} ) \leq - 4 \gamma \cL_\infty ( \Theta_{m} ) \leq 0.
\end{equation}
\Nobs that \eqref{eq:induction:1} shows that $V(\Theta_n) \leq V(\Theta_{n-1}) \leq \cdots  \leq V(\Theta_0)$. The assumption that $\gamma \leq (4 V ( \Theta_0 ) + 2 ) ^{-1}$ hence ensures that $\gamma \leq (4 V ( \Theta_0) + 2 )^{-1} \leq ( 4 V ( \Theta_n) + 2 )^{-1}$. Combining this and \cref{cor:est:vtheta_n} demonstrates that
\begin{equation}
    \begin{split}
    V(\Theta_{n+1}) - V(\Theta_n) &\leq \rbr*{  - 8 \gamma + 8\gamma ^2 ( 2 V(\Theta_n) + 1 )  }  \cL_\infty ( \Theta _ n) \\
    &\leq  \rbr*{  - 8 \gamma + 8 \gamma \br*{ \tfrac{2 V(\Theta_n) + 1}{4 V(\Theta_n) + 2} }  } \cL_\infty ( \Theta _ n)  = - 4 \gamma \cL_\infty (\Theta_n) \leq 0.
    \end{split}
\end{equation}
This completes the proof of \cref{lem:vthetan:decreasing}.
\end{proof}

\subsection{Convergence of the risks of gradient descent processes}
\label{subsection:theorem:gd}

\begin{theorem} \label{theo:gd:loss} 
Assume \cref{setting:const}, let $\gamma \in (0, \infty)$,  and let $\Theta = (\Theta_n)_{n \in \N_0}  \colon \N_0 \to \R^{3 \width + 1}$ satisfy for all $n \in \N_0$ that $\Theta_{n+1} = \Theta_n - \gamma \cG ( \Theta_n)$ and  $\gamma \leq (4 V ( \Theta_0) + 2 )^{-1}$. Then 
\begin{enumerate} [(i)] 
    \item \label{theo:gd:item1} it holds that $\sup_{n \in \N_0} \norm{ \Theta_n } \leq \br{ V( \Theta_0 ) } ^{1/2} < \infty$ and
    \item \label{theo:gd:item2} it holds that $\limsup_{n \to \infty} \cL_\infty (\Theta_n) = 0$.
\end{enumerate}
\end{theorem}
\begin{proof} [Proof of \cref{theo:gd:loss}]
\Nobs that \cref{lem:vthetan:decreasing} proves that for all $n \in \N_0$ we have that $V(\Theta_n ) \leq V ( \Theta_{n-1}) \leq \cdots \leq V(\Theta_0)$. This and the fact that $\forall \, n \in \N_0 \colon \norm{ \Theta_n } \leq \br{ V ( \Theta_n ) }^{1/2}$ establish item \eqref{theo:gd:item1}. Next \nobs that \cref{lem:vthetan:decreasing} implies for all ${N \in \N}$ that
\begin{equation}
      \sum_{n=0}^{N - 1} \rbr[\big]{ 4 \gamma \cL_\infty (\Theta_n ) } \leq \sum_{n = 0}^{N - 1} \rbr[\big]{  V(\Theta_{n}) - V(\Theta_{n+1}) } = V(\Theta_0) - V( \Theta_N) \leq V(\Theta_0).
\end{equation}
Hence, we have that
\begin{equation}
    \sum_{n=0}^\infty  \cL_\infty (\Theta_n) \leq \frac{V ( \Theta_0 )}{4 \gamma} < \infty.
\end{equation}
This shows that $\limsup_{n \to \infty} \cL_\infty ( \Theta_n ) = 0$.
The proof of \cref{theo:gd:loss} is thus complete.
\end{proof}

\begin{cor} \label{cor:gd:main}
Assume \cref{setting:const}, let $\gamma \in (0, \infty)$,  and let $\Theta = (\Theta_n)_{n \in \N_0}  \colon \N_0 \to \R^{3 \width + 1}$ satisfy for all $n \in \N_0$ that $\Theta_{n+1} = \Theta_n - \gamma \cG ( \Theta_n)$ and  $\gamma \leq  \br{ 12 \norm{ \Theta_0 } ^2 + 32 \alpha ^2 + 2 }^{-1}$. Then 
\begin{enumerate} [(i)] 
    \item it holds that $\sup_{n \in \N_0} \norm{ \Theta_n } \leq \br{ V ( \Theta_0 ) }^{1/2} < \infty$ and
    \item it holds that $\limsup_{n \to \infty} \cL_\infty (\Theta_n) = 0$.
\end{enumerate}
\end{cor}
\begin{proof} [Proof of \cref{cor:gd:main}]
\Nobs that \cref{prop:lyapunov:norm} proves that $ 4 V ( \Theta_0 ) + 2 \leq 12 \norm{ \Theta_0 } ^2 + 32 \alpha ^2 + 2  $. Hence, we have that $\gamma \leq (4 V ( \Theta_0) + 2 )^{-1}$. Combining this with \cref{theo:gd:loss} completes the proof of \cref{cor:gd:main}.
\end{proof}

\subsection{Gradient descent processes with random initializations}
\label{subsection:gd:random:initialization}

\begin{cor} \label{cor:gd:random}
Assume \cref{setting:const}, let $c, \gamma \in (0, \infty)$, let $(\Omega, \cF, \P)$ be a probability space, let $\Theta = ( \Theta_n)_{n \in \N_0} \colon \Omega \times \N_0 \to \R^{3 \width + 1}$ be a stochastic process, assume $\Theta_0(\Omega) \subseteq [-c ,c ]^{3 \width + 1 }$, assume for all $n \in \N_0$ that $\Theta_{n+1} = \Theta_n - \gamma \cG ( \Theta_n )$, and assume $\gamma \leq \br{12 c^2 (3 \width + 1) + 32 \alpha ^2 + 2}^{-1}$. Then
\begin{enumerate} [(i)]
\item \label{cor:gd:random:item1}
it holds that $\sup_{\omega \in \Omega} \sup_{n \in \N_0} \norm{ \Theta_n ( \omega)} \leq \br{3c^2 (3 \width + 1 ) + 8 \alpha ^2 } ^{1/2} < \infty$,
    \item \label{cor:gd:random:item2}
    it holds for all $\omega \in \Omega$ that $\limsup_{n \to \infty} \cL_\infty (\Theta_n ( \omega)) = 0$, and
    \item \label{cor:gd:random:item3}
    it holds that $\limsup_{n \to \infty} \E [ \cL_\infty (\Theta_n) ] = 0$.
\end{enumerate}
\end{cor}
\begin{proof} [Proof of \cref{cor:gd:random}]
\Nobs that \cref{prop:lyapunov:norm} demonstrates for all $\phi \in [-c , c]^{3 \width + 1}$ that
\begin{equation} \label{eq:gd:random:1}
    V(\phi) \leq 3 \| \phi \| ^2 + 8 \alpha ^2 \leq 3 c^2 ( 3 \width + 1) + 8 \alpha ^2.
\end{equation}
Hence, we have for all $\phi \in [-c , c]^{3 \width + 1}$ that
\begin{equation}
    \gamma \leq  \br{12 c^2 (3 \width + 1) + 32 \alpha ^2 + 2}^{-1} \leq \br{4 V ( \phi ) + 2}^{-1}.
\end{equation}
This demonstrates for all $\omega \in \Omega$ that $\gamma \leq ( 4 V(\Theta_0 ( \omega)) + 2)^{-1}$. 
\cref{lem:vthetan:decreasing} and \eqref{eq:gd:random:1} hence prove that for all $\omega \in \Omega$, $n \in \N_0$ we have that $\norm{\Theta_n ( \omega)} \leq \br{ V ( \Theta_n ( \omega)) }^{1/2} \leq \br{V ( \Theta_0 ( \omega)) }^{1/2} \leq  \br{3c^2 ( 3 \width + 1 ) + 8 \alpha^2 }^{1/2}$. This establishes item \eqref{cor:gd:random:item1}. Next \nobs that \cref{theo:gd:loss} shows for all $\omega \in \Omega$ that $\limsup_{n \to \infty} \cL_\infty (\Theta_n ( \omega)) = 0 $, which proves item \eqref{cor:gd:random:item2}. Furthermore, \nobs that \cref{lem:vthetan:decreasing} assures that for all $\omega \in \Omega$, $N \in \N$ it holds that
\begin{equation}
    \sum_{n=0}^{N-1}  \rbr[\big]{ 4 \gamma \cL_\infty ( \Theta_n ( \omega )) } \leq \sum_{n=0}^{N-1} \rbr[\big]{V(\Theta_{n+1} ( \omega) ) - V ( \Theta_n ( \omega )) } \leq V ( \Theta_0 ( \omega )).
\end{equation}
\cref{prop:lyapunov:norm} hence shows that for all $\omega \in \Omega$ we have that
\begin{equation}
    \sum_{n=0}^\infty  \cL_\infty (\Theta_n ( \omega) ) \leq \frac{V ( \Theta_0 ( \omega ))}{4\gamma} \leq \frac{3 \norm{\Theta_0 ( \omega ) } ^2 + 8 \alpha ^2}{4 \gamma}.
\end{equation}
Combining this, item \eqref{cor:gd:random:item2}, and the dominated convergence theorem establishes item \eqref{cor:gd:random:item3}.
The proof of \cref{cor:gd:random} is thus complete.
\end{proof}

\section{A priori estimates for general target functions}
\label{section:apriori:gen}
The key ingredient in our convergence proofs for gradient flow and GD processes in Sections \ref{section:gradientflow} and \ref{section:gradientdescent} are suitable a priori estimates for the gradient flow processes (see \cref{lem:flow:lyapunov} in \cref{subsection:ito:lyapunov}) and the GD processes (see \cref{lem:vthetan:decreasing} in \cref{subsection:gd:lyapunov}). To initiate further research activities of this kind, we derive in this section related a priori bounds in the case of general target functions. For details we refer to \eqref{prop:gen:apriori:eq1} and \eqref{prop:gen:apriori:eq2} in \cref{prop:gen:apriori} below.

\begin{prop} \label{prop:gen:apriori} 
Let $\width \in \N$, $f \in C ( [0, 1 ], \R)$,
let $\fw = (( \w{\phi} _ 1 , \ldots, \w{\phi} _ \width ))_{ \phi \in \R^{3 \width + 1}} \colon \R^{3 \width + 1} \to \R^{\width}$,
$\fb =  (( \b{\phi} _ 1 , \ldots, \b{\phi} _ \width ))_{ \phi \in \R^{3 \width + 1}} \colon \R^{3 \width + 1} \to \R^{\width}$,
$\fv =  (( \v{\phi} _ 1 , \ldots, \v{\phi} _ \width ))_{ \phi \in \R^{3 \width + 1}} \colon \R^{3 \width + 1} \to \R^{\width}$, 
$\fc = (\c{\phi})_{\phi \in \R^{3 \width + 1 }} \colon \R^{3 \width + 1} \to \R$,
and $\norm{ \cdot} \colon \R^{3 \width + 1 } \to [0, \infty)$
 satisfy for all $\phi  = ( \phi_1 ,  \ldots, \phi_{3 \width + 1}) \in \R^{3 \width + 1}$, $j \in \{1, 2, \ldots, \width \}$ that
 $\w{\phi}_j = \phi_j$, $\b{\phi}_j = \phi_{\width + j}$, 
$\v{\phi}_j = \phi_{2\width + j}$, $\c{\phi} = \phi_{3 \width + 1}$,
and $\norm{ \phi } = [ \sum_{i=1}^{3 \width + 1 } | \phi_i | ^2 ] ^{ 1/2 }$,
let $\scrN = (\realization{\phi})_{\phi \in \R^{3 \width + 1}} \colon \R^{3 \width + 1 } \to  C(\R , \R)$ and $\cL \colon \R^{3 \width + 1 } \to \R$ satisfy for all $\phi \in \R^{3 \width + 1}$, $x \in \R$ that $\realization{\phi} (x) = \c{\phi} + \sum_{j=1}^\width \v{\phi}_j \max \{ \w{\phi}_j x + \b{\phi}_j , 0 \}$ and $\cL (\phi) = \int_0^1 (\realization{\phi} (y) - f ( y ) )^2 \d y$,
let $V \colon \R^{3 \width + 1 } \to \R$ and $\cG = (\cG_1 , \ldots, \cG_{3 \width +1} ) \colon \R^{3 \width + 1 } \to \R^{3 \width + 1}$ satisfy for all $\phi \in \R^{3 \width + 1}$, $j \in \{1,2, \ldots, \width \}$ that $V ( \phi ) = \norm{\phi} ^2 + \abs{ \c{\phi} } ^2$ and
\begin{equation} \label{eq:defgradient:gen} 
\begin{split}
        \cG_j ( \phi) &= 2\v{\phi}_j \int_{0}^1 x ( \realization{\phi} (x) - f ( x )) \indicator{(0, \infty )} ( \w{\phi}_j x + \b{\phi}_j) \d x, \\
        \cG_{\width + j} ( \phi) &= 2 \v{\phi}_j \int_{0}^1 (\realization{\phi} (x) - f ( x ) ) \indicator{(0, \infty )} ( \w{\phi}_j x + \b{\phi}_j) \d x, \\
        \cG_{2 \width + j} ( \phi) &= 2 \int_0^1 [\max \{ \w{\phi}_j x + \b{\phi}_j , 0 \}] ( \realization{\phi}(x) - f ( x )) \d x, \\
        \cG_{3 \width + 1} ( \phi) &= 2 \int_0^1 (\realization{\phi} (x) - f ( x ) ) \d x,
        \end{split}
\end{equation}
and let $\Theta \in C([0, \infty) , \R^{3 \width + 1})$ satisfy for all $t \in [0, \infty)$ that $\Theta_t = \Theta_0 - \int_0^t \cG ( \Theta_s ) \d s$.
Then 
\begin{enumerate} [(i)]
    \item \label{prop:apriori:item1} it holds for all $t \in [ 0, \infty)$ that
\begin{equation} \label{prop:gen:apriori:eq1}
V( \Theta_t ) = V ( \Theta_0 ) - 8 \int_0^t \int_0^1 \realization{\Theta_s} ( x ) ( \realization{\Theta_s} ( x ) - f ( x ) ) \d x \d s \leq V ( \Theta_0 ) + 2 t \int_0^1 \abs{ f(x) } ^2 \d x
\end{equation}
and
\item \label{prop:apriori:item2} it holds for all $t \in [0, \infty)$ that
\begin{equation} \label{prop:gen:apriori:eq2}
    \norm{ \Theta_t } \leq ( V ( \Theta_0 ) )^{1/2} + \br*{2 \textstyle\int_0^1 \abs{  f (x) } ^2 \d x}^{1/2}  t^{1/2} .
\end{equation}
\end{enumerate}
\end{prop}
\begin{proof}[Proof of \cref{prop:gen:apriori}]
Throughout this proof let $\langle \cdot , \cdot \rangle \colon \R^{3 \width + 1 } \times \R^{3 \width + 1 } \to \R$ satisfy for all $\phi = ( \phi_1 , \ldots , \phi_{3 \width + 1 })$, $\psi = ( \psi_1 , \ldots , \psi_{3 \width + 1 } ) \in \R^{3 \width + 1 }$ that $\langle \phi , \psi \rangle = \sum_{i=1}^{3 \width + 1 } \phi_i \psi_i$.
\Nobs that for all $\phi \in \R^{3 \width + 1}$ it holds that
\begin{equation}
    (\nabla V) ( \phi ) = 2 \rbr[\big]{  \w{\phi}_1, \w{\phi}_2 , \ldots, \w{\phi}_\width , \b{\phi}_1 , \b{\phi}_2 , \ldots, \b{\phi}_\width , \v{\phi}_1, \v{\phi}_2, \ldots, \v{\phi}_{\width},  2 \c{\phi}  } .
\end{equation} 
This implies for all $\phi \in \R^{3 \width + 1}$ that
\begin{equation} \label{apriori:gen:eq1}
    \begin{split}
        &\langle ( \nabla V ) ( \phi) , \cG(\phi) \rangle \\
        &= 4 \br[\Bigg]{ \sum_{j=1}^\width \v{\phi}_j \int_0^1 [\max \{ \w{\phi}_j x + \b{\phi}_j , 0 \}] ( \realization{\phi} (x) - f (x) ) \d x }
        + 8 \c{\phi} \br*{ \int_0^1 (\realization{\phi} (x) - f (x) ) \d x } \\
        &+ 4 \br[\Bigg]{  \sum_{j=1}^\width \v{\phi}_j \int_{0}^1 (\w{\phi}_j x + \b{\phi}_j) (\realization{\phi} (x) - f (x)) \indicator{(0, \infty )} ( \w{\phi}_j x + \b{\phi}_j) \d x } \\
        & = 8 \br*{\int_0^1 \rbr*{  \smallsum_{j=1}^\width \v{\phi}_j [ \max \{ \w{\phi}_j  x + \b{\phi}_j , 0 \} ]  }(\realization{\phi} ( x ) -  f (x) ) \d x } 
        + 8 \c{\phi} \br*{ \int_0^1 (\realization{\phi} (x) - f (x) ) \d x } \\
        &= 8 \int_0^1 \realization{\phi}(x) (\realization{\phi} (x) - f (x) ) \d x.
    \end{split}
\end{equation}
Next \nobs that the fact that for all $x , y \in \R$ it holds that $x ( x - y ) =  (x - \frac{y}{2})^2 -\frac{1}{4} y ^2 \geq -\frac{1}{4} y ^2$ ensures that for all $x \in [0,1]$, $\phi \in \R^{3 \width + 1}$ it holds that $\realization{\phi} (x) (\realization{\phi} (x) - f (x) ) \geq - \frac{1}{4} (f (x))^2$. Hence, we have for all $\phi \in \R^{3 \width + 1}$ that
\begin{equation}
      \langle ( \nabla V ) ( \phi) , \cG(\phi) \rangle \geq - 2 \int_0^1 | f (x) | ^2 \d x.
\end{equation}
This, \eqref{apriori:gen:eq1}, the fact that $V \in C^\infty ( \R^{3 \width + 1 }, \R)$, and \cref{lem:chainrule:gen} shows for all $t \in [ 0, \infty)$ that
\begin{equation}
\begin{split}
    V ( \Theta_t ) - V ( \Theta_0 ) 
    &= - \int_0^t \langle ( \nabla V ) ( \Theta_s ) , \cG ( \Theta_s) \rangle \d s \\
    &= - 8 \int_0^t \int_0^1 \realization{\Theta_s} ( x ) ( \realization{\Theta_s} ( x ) - f ( x ) ) \d x \d s \\
    &\leq 2 \int_0^t \int_0^1 | f (x) | ^2 \d x \d s 
    = 2 t \int_0^1 | f (x) | ^2 \d x.
    \end{split}
\end{equation}
This proves item \eqref{prop:apriori:item1}. Next \nobs that item \eqref{prop:apriori:item1} and the fact that $\forall \, \phi \in \R^{3 \width + 1 } \colon \norm{ \phi } ^2 \leq V ( \phi)$ demonstrate that for all $t \in [0, \infty)$ it holds that
\begin{equation}
    \norm{ \Theta_t } \leq ( V ( \Theta_t ) )^{1/2} \leq \br*{ V ( \Theta_0 ) + 2 t \textstyle\int_0^1 \abs{  f (x) } ^2 \d x}^{1/2}.
\end{equation}
Combining this and the fact that $\forall \, x , y \in [0, \infty) \colon (x + y )^{1/2} \leq x ^{1/2} + y ^{1/2}$ ensures that for all $t \in [0, \infty)$ we have that 
\begin{equation}
    \norm{ \Theta_t } \leq (V ( \Theta_0 ) )^{1/2} +  \br*{2 \textstyle\int_0^1 \abs{  f (x) } ^2 \d x}^{1/2}   t^{1/2}.
\end{equation}
This establishes item \eqref{prop:apriori:item2}.
The proof of \cref{prop:gen:apriori} is thus complete.
\end{proof}

\subsection*{Acknowledgments}
Benno Kuckuck is gratefully acknowledged for several helpful suggestions.
This work has been funded by the Deutsche Forschungsgemeinschaft (DFG, German Research Foundation) under Germany’s Excellence Strategy EXC 2044-390685587, Mathematics Münster: Dynamics-Geometry-Structure.


\begin{thebibliography}{10}

\bibitem{AkyildizSabanis2021}
{\sc Akyildiz, {\"O}.~D., and Sabanis, S.}
\newblock Nonasymptotic analysis of {S}tochastic {G}radient {H}amiltonian
  {M}onte {C}arlo under local conditions for nonconvex optimization.
\newblock {\em arXiv:2002.05465\/} (2021), 26 pages.

\bibitem{AllenzhuLiLiang2019}
{\sc Allen-Zhu, Z., Li, Y., and Liang, Y.}
\newblock Learning and generalization in overparameterized neural networks,
  going beyond two layers.
\newblock In {\em Advances in Neural Information Processing Systems\/} (2019),
  H.~Wallach, H.~Larochelle, A.~Beygelzimer, F.~d\textquotesingle
  Alch\'{e}-Buc, E.~Fox, and R.~Garnett, Eds., vol.~32, Curran Associates,
  Inc., pp.~6158--6169.

\bibitem{AllenzhuLiSong2019}
{\sc Allen-Zhu, Z., Li, Y., and Song, Z.}
\newblock A convergence theory for deep learning via over-parameterization.
\newblock In {\em Proceedings of the 36th International Conference on Machine
  Learning\/} (09--15 Jun 2019), K.~Chaudhuri and R.~Salakhutdinov, Eds.,
  vol.~97 of {\em Proceedings of Machine Learning Research}, PMLR,
  pp.~242--252.

\bibitem{Bach2017}
{\sc Bach, F.}
\newblock Breaking the curse of dimensionality with convex neural networks.
\newblock {\em Journal of Machine Learning Research 18}, 19 (2017), 1--53.

\bibitem{BachMoulines2013}
{\sc Bach, F., and Moulines, E.}
\newblock Non-strongly-convex smooth stochastic approximation with convergence
  rate {$O(1/n)$}.
\newblock In {\em Advances in Neural Information Processing Systems\/} (2013),
  C.~J.~C. Burges, L.~Bottou, M.~Welling, Z.~Ghahramani, and K.~Q. Weinberger,
  Eds., vol.~26, Curran Associates, Inc., pp.~773--781.

\bibitem{CheriditoJentzenRossmannek2020}
{\sc Cheridito, P., Jentzen, A., and Rossmannek, F.}
\newblock Non-convergence of stochastic gradient descent in the training of
  deep neural networks.
\newblock {\em Journal of Complexity\/} (2020), 101540.

\bibitem{ChizatBach2018}
{\sc Chizat, L., and Bach, F.}
\newblock On the global convergence of gradient descent for over-parameterized
  models using optimal transport.
\newblock In {\em Advances in Neural Information Processing Systems\/} (2018),
  S.~Bengio, H.~Wallach, H.~Larochelle, K.~Grauman, N.~Cesa-Bianchi, and
  R.~Garnett, Eds., vol.~31, Curran Associates, Inc., pp.~3036--3046.

\bibitem{DuLeeLiWangZhai2019}
{\sc Du, S., Lee, J., Li, H., Wang, L., and Zhai, X.}
\newblock Gradient descent finds global minima of deep neural networks.
\newblock In {\em Proceedings of the 36th International Conference on Machine
  Learning\/} (Long Beach, California, USA, 6 2019), K.~Chaudhuri and
  R.~Salakhutdinov, Eds., vol.~97 of {\em Proceedings of Machine Learning
  Research}, PMLR, pp.~1675--1685.

\bibitem{DuZhaiPoczosSingh2018arXiv}
{\sc Du, S.~S., Zhai, X., Pocz\'{o}s, B., and Singh, A.}
\newblock Gradient descent provably optimizes over-parameterized neural
  networks.
\newblock {\em arXiv:1810.02054\/} (2018), 19 pages.

\bibitem{EMaWu2020}
{\sc E, W., Ma, C., and Wu, L.}
\newblock A comparative analysis of optimization and generalization properties
  of two-layer neural network and random feature models under gradient descent
  dynamics.
\newblock {\em Science China Mathematics 63}, 7 (2020), 1235--1258.

\bibitem{FehrmanGessJentzen2020}
{\sc Fehrman, B., Gess, B., and Jentzen, A.}
\newblock Convergence rates for the stochastic gradient descent method for
  non-convex objective functions.
\newblock {\em Journal of Machine Learning Research 21}, 136 (2020), 1--48.

\bibitem{Hanin2018}
{\sc Hanin, B.}
\newblock Which neural net architectures give rise to exploding and vanishing
  gradients?
\newblock In {\em Advances in Neural Information Processing Systems\/} (2018),
  S.~Bengio, H.~Wallach, H.~Larochelle, K.~Grauman, N.~Cesa-Bianchi, and
  R.~Garnett, Eds., vol.~31, Curran Associates, Inc., pp.~582--591.

\bibitem{HaninRolnick2018}
{\sc Hanin, B., and Rolnick, D.}
\newblock How to start training: The effect of initialization and architecture.
\newblock In {\em Advances in Neural Information Processing Systems\/} (2018),
  S.~Bengio, H.~Wallach, H.~Larochelle, K.~Grauman, N.~Cesa-Bianchi, and
  R.~Garnett, Eds., vol.~31, Curran Associates, Inc., pp.~571--581.

\bibitem{JacotGabrielHongler2020}
{\sc Jacot, A., Gabriel, F., and Hongler, C.}
\newblock Neural tangent kernel: Convergence and generalization in neural
  networks.
\newblock In {\em Advances in Neural Information Processing Systems\/} (2018),
  S.~Bengio, H.~Wallach, H.~Larochelle, K.~Grauman, N.~Cesa-Bianchi, and
  R.~Garnett, Eds., vol.~31, Curran Associates, Inc., pp.~8571--8580.

\bibitem{JentzenKuckuckNeufeldVonWurstemberger2021}
{\sc Jentzen, A., Kuckuck, B., Neufeld, A., and von Wurstemberger, P.}
\newblock Strong error analysis for stochastic gradient descent optimization
  algorithms.
\newblock {\em IMA Journal of Numerical Analysis 41}, 1 (2021), 455--492.

\bibitem{JentzenvonWurstemberger2020}
{\sc Jentzen, A., and {von Wurstemberger}, P.}
\newblock Lower error bounds for the stochastic gradient descent optimization
  algorithm: Sharp convergence rates for slowly and fast decaying learning
  rates.
\newblock {\em Journal of Complexity 57\/} (2020), 101438.

\bibitem{LeiHuLiTang2020}
{\sc {Lei}, Y., {Hu}, T., {Li}, G., and {Tang}, K.}
\newblock Stochastic gradient descent for nonconvex learning without bounded
  gradient assumptions.
\newblock {\em IEEE Transactions on Neural Networks and Learning Systems 31},
  10 (2020), 4394--4400.

\bibitem{LiLiang2019}
{\sc Li, Y., and Liang, Y.}
\newblock Learning overparameterized neural networks via stochastic gradient
  descent on structured data.
\newblock In {\em Advances in Neural Information Processing Systems\/} (2018),
  S.~Bengio, H.~Wallach, H.~Larochelle, K.~Grauman, N.~Cesa-Bianchi, and
  R.~Garnett, Eds., vol.~31, Curran Associates, Inc., pp.~8157--8166.

\bibitem{LovasSabanis2020}
{\sc Lovas, A., Lytras, I., Rásonyi, M., and Sabanis, S.}
\newblock Taming neural networks with {TUSLA}: Non-convex learning via adaptive
  stochastic gradient {L}angevin algorithms.
\newblock {\em arXiv:2006.14514\/} (2020), 29 pages.

\bibitem{LuShinSuKarniadakis2020}
{\sc Lu, L., Shin, Y., Su, Y., and Karniadakis, G.~E.}
\newblock Dying {R}e{LU} and initialization: Theory and numerical examples.
\newblock {\em Communications in Computational Physics 28}, 5 (2020),
  1671--1706.

\bibitem{BachMoulines2011}
{\sc Moulines, E., and Bach, F.}
\newblock Non-asymptotic analysis of stochastic approximation algorithms for
  machine learning.
\newblock In {\em Advances in Neural Information Processing Systems\/} (2011),
  J.~Shawe-Taylor, R.~Zemel, P.~Bartlett, F.~Pereira, and K.~Q. Weinberger,
  Eds., vol.~24, Curran Associates, Inc., pp.~451--459.

\bibitem{Ruder2017overview}
{\sc Ruder, S.}
\newblock An overview of gradient descent optimization algorithms.
\newblock {\em arXiv:1609.04747\/} (2017), 14 pages.

\bibitem{SankararamanDeXuHuangGoldstein2020}
{\sc Sankararaman, K.~A., De, S., Xu, Z., Huang, W.~R., and Goldstein, T.}
\newblock The impact of neural network overparameterization on gradient
  confusion and stochastic gradient descent.
\newblock {\em arXiv:1904.06963\/} (2020), 28 pages.

\bibitem{ShinKarniadakis2020}
{\sc Shin, Y., and Karniadakis, G.~E.}
\newblock Trainability of {R}e{LU} networks and data-dependent initialization.
\newblock {\em Journal of Machine Learning for Modeling and Computing 1}, 1
  (2020), 39--74.

\bibitem{EChaoWu2018}
{\sc Wu, L., Ma, C., and E, W.}
\newblock How {SGD} selects the global minima in over-parameterized learning: A
  dynamical stability perspective.
\newblock In {\em Advances in Neural Information Processing Systems\/} (2018),
  S.~Bengio, H.~Wallach, H.~Larochelle, K.~Grauman, N.~Cesa-Bianchi, and
  R.~Garnett, Eds., vol.~31, Curran Associates, Inc., pp.~8279--8288.

\bibitem{ZouCaoZhouGu2019}
{\sc Zou, D., Cao, Y., Zhou, D., and Gu, Q.}
\newblock Gradient descent optimizes over-parameterized deep {R}e{LU} networks.
\newblock {\em Machine Learning 109\/} (2020), 467--492.

\end{thebibliography}
\end{document}